\documentclass[a4paper, 10pt, reqno]{amsart}
\usepackage{amsmath}
\usepackage{amssymb}
\usepackage{amsthm}
\usepackage{mathtools, verbatim}
\usepackage[utf8]{inputenc}
\usepackage{indentfirst}
\usepackage{enumerate}
\usepackage[colorlinks=true,hidelinks]{hyperref}
\hypersetup{urlcolor=blue, citecolor=red}
\usepackage{graphicx}
\usepackage[usenames,dvipsnames]{pstricks}
\usepackage{epsfig}
\usepackage{pst-grad} 
\usepackage{pst-plot} 
\usepackage{ifthen}
\usepackage{color,xcolor}
\usepackage[numbers]{natbib}

\def\sideremark#1{\ifvmode\leavevmode\fi\vadjust{\vbox to0pt{\vss
 \hbox to 0pt{\hskip\hsize\hskip1em
 \vbox{\hsize2.1cm\tiny\raggedright\pretolerance10000
  \noindent #1\hfill}\hss}\vbox to15pt{\vfil}\vss}}}%

\allowdisplaybreaks

\usepackage{fullpage}
\numberwithin{equation}{section}

\def\polhk#1{\setbox0=\hbox{#1}{\ooalign{\hidewidth\lower1.5ex\hbox{`}\hidewidth\crcr\unhbox0}}}

\def\Xint#1{\mathchoice
{\XXint\displaystyle\textstyle{#1}}%
{\XXint\textstyle\scriptstyle{#1}}%
{\XXint\scriptstyle\scriptscriptstyle{#1}}%
{\XXint\scriptscriptstyle\scriptscriptstyle{#1}}%
\!\int}
\def\XXint#1#2#3{{\setbox0=\hbox{$#1{#2#3}{\int}$ }
\vcenter{\hbox{$#2#3$ }}\kern-.6\wd0}}

\def\dashint{\Xint-}

\newcommand{\R}{\mathbb{R}}
\newcommand{\N}{\mathbb{N}}

\theoremstyle{plain}
\newtheorem{theorem}{Theorem}[section]

\newtheorem{definition}{Definition}[section]
\newtheorem{lemma}{Lemma}[section]
\newtheorem{corollary}{Corollary}[section]
\newtheorem{proposition}{Proposition}[section]
\newtheorem{remark}{Remark}[section]
\newtheorem{Assumption}{A}

\title{Global fractional Sobolev regularity for fully nonlinear elliptic equations}

\author{Claudemir Alcantara and Makson Santos}
 
\begin{document}


  
\begin{abstract} 
We investigate fractional regularity estimates up to the boundary for solutions to fully nonlinear elliptic equations with measurable ingredients. Specifically, under the assumption of uniform ellipticity of the operator, we demonstrate that viscosity solutions to a second-order operator satisfy a fractional Laplacian equation. This result implies that the solutions are globally of class $W^{\gamma, p}$, for $\gamma \in (1,2)$ with appropriate estimates. Consequently, these solutions exhibit differentiability of order strictly greater than one, without requiring any additional assumptions regarding the operator, such as convexity or concavity.    
\end{abstract}   
 
\date{\today}

\maketitle
 
\tableofcontents


\section{Introduction}\label{sct intro}

In this paper, we study the regularity theory of viscosity solutions to fully nonlinear equations of the form
\begin{equation}\label{1}
\begin{cases}
F(D^2u,Du, u,x) = f(x) & \quad \text{in} \quad \Omega, \\
u = \varphi & \quad \text{on} \quad \partial\Omega,
\end{cases}
\end{equation}
where $F: \mathcal{S}(d) \times \mathbb{R}^d \times \mathbb{R} \times B^+_1 \rightarrow \mathbb{R}$ is a $(\lambda, \Lambda)$-elliptic operator, $f \in L^p(\Omega)$, and $\varphi \in W^{2,p}(\Omega)$. Here, $\mathcal{S}(d)$ denotes the set of symmetric matrices. In particular, we establish boundary estimates in the fractional Sobolev space $W^{\gamma,p}(B^+_{1/2})$ for some universal constant $\gamma \in (1,2)$. These, combined with interior regularity results, lead to global estimates.

Over the past few decades, significant progress has been made in the regularity theory of fully nonlinear equations, which plays a crucial role in mathematical analysis. The foundational work began with the result in \cite{Krylov-Safonov1980}, where the author proved that solutions to
\begin{equation}\label{eq_homogeq}
F(D^2u) = 0 \quad \text{in} \quad B_1
\end{equation}
are locally of class $C^\alpha$ for some universal, but unknown, $\alpha \in (0,1)$. Later, gradient estimates were derived in \cite{Trudinger88}, where a linearization approach showed that the gradient of the solutions satisfies a linear equation, leading to $C^{\alpha_0}$-estimates for the gradient. Consequently, this implies that solutions are locally of class $C^{1,\alpha_0}$ for some unknown $\alpha_0 \in (0,1)$. Throughout this manuscript, we use $\alpha_0$ to represent the Hölder exponent corresponding to the $C^{1, \alpha_0}$-regularity of solutions to \eqref{eq_homogeq}.

Once gradient estimates were established, the next question was whether classical solutions could be obtained. In \cite{Evans82, Krylov83} this was answered, where the authors independently showed that, under the additional assumption of convexity on the operator $F$, solutions belong to the class $C^{2,\alpha}_{\text{loc}}$, confirming the existence of classical solutions. This result became known as the Evans-Krylov theory.

In \cite{Caffarelli1989}, the author studied fully nonlinear elliptic equations with non-homogeneous and variable coefficients of the form
\begin{equation}\label{eq_main3}
F(D^2u, x) = f(x) \quad \text{in} \, B_1,
\end{equation}
where $f \in L^p(B_1) \cap C(B_1)$ for $p > d$. The author established $C^{1, \alpha}_{\text{loc}}$, $C^{2, \alpha}_{\text{loc}}$, and $W^{2, p}_{\text{loc}}$-estimates, using approximation methods to link the regularity of solutions to \eqref{eq_main3} with the regularity of solutions to \eqref{eq_homogeq}. While the $C^{1, \alpha}$-estimates only require uniform ellipticity and a small oscillation condition on $F$ (see Definition \ref{def_oscF} and Assumption \ref{A-oscF}), the key assumption for proving $W^{2, p}_{\text{loc}}$ regularity is the $C^{1,1}$-estimates for solutions to the homogeneous equation \eqref{eq_homogeq}. Specifically, the author assumed that solutions to \eqref{eq_homogeq} are in $C^{1,1}_{\text{loc}}(B_1)$, with the estimate
\[
\|u\|_{C^{1,1}(B_{1/2})} \leq C\|u\|_{L^\infty(B_1)},
\]
where $C$ is a universal constant. This condition, for example, holds when the operator is convex or concave, as ensured by the Evans-Krylov theory. In terms of H\"{o}lder regularity, $C^{1, \alpha}_{\text{loc}}$-estimates are the best one can expect under the assumption of uniform ellipticity alone, at least for $d \geq 5$. We refer to the counterexamples in \cite{Nad-Vla13}, where the authors show that for any constants $0 < \lambda \leq \Lambda$ and $\beta \in (0,1)$, it is possible to construct an operator $F$ whose solutions do not belong to $C^{1, \beta}$.

Further advancements were made in \cite{Escauriaza93}, where the author refined the $W^{2,p}$-estimates from \cite{Caffarelli1989} by allowing $p$ to be slightly less than the dimension $d$. Namely, the condition $p > d$ was replaced with $p > d - \varepsilon_0$, where $\varepsilon_0 \in (0, d/2)$ is known as the {\it Escauriaza exponent}, and it depends only on the dimension $d$ and the ellipticity constants $\lambda$ and $\Lambda$. For $p \in (d - \varepsilon_0, d)$, and under the assumption of uniform ellipticity for the operator $F$, the author in \cite{Swiech97} proved that solutions to \eqref{1} belong to the space $W^{1, q}_{\text{loc}}$, where
\[
q < p^* := \dfrac{dp}{d - p}, \quad \text{with} \quad d^* = \infty.
\]
However, this regularity is not optimal, at least regarding differentiability. In \cite{Pimentel-Santos-Teixeira2022}, the authors showed that, under the assumption of uniform ellipticity alone, solutions to 
\begin{equation}\label{eq_main4}
F(D^2u) = f(x) \quad \text{in} \, B_1,
\end{equation}
where $f \in L^p(B_1)$ and $p \in (d - \varepsilon_0, +\infty)$, are locally of class $W^{1+\varepsilon, p}$, for any $\varepsilon \in (0, \alpha_0)$. Unlike the paraboloid-based approach in \cite{Caffarelli1989}, \cite{Pimentel-Santos-Teixeira2022} introduced the use of $C^{1, \alpha}$-cones (see Definition \ref{def_cones} below), which allowed them to apply the $C^{1, \alpha_0}_{\text{loc}}$-estimates already known for solutions to \eqref{eq_homogeq}, without needing additional assumptions like convexity. By employing this new approach, the authors demonstrated that viscosity solutions to \eqref{eq_main4} are also weak solutions to
\begin{equation}
\begin{cases}\label{eq_fraclapintro}
    (-\Delta)^{\frac{\sigma}{2}} u = g \quad &\text{in} \, B_1,\\
    \tilde{u} = 0 \quad &\text{in} \, \R^d \setminus B_1,
\end{cases}
\end{equation}
where $\sigma = 1 + \varepsilon$, thus achieving the desired regularity.

Boundary regularity has also been an active area of research. In \cite{Krylov83}, the author extended the Evans-Krylov theory to global estimates for solutions to \eqref{1} with $f \equiv 0$. For equations with a right-hand side in $C^\alpha$ and Dirichlet boundary conditions, the author in \cite{Silvestre-Sirakov2014} established $C^{1,\alpha}$ and $C^{2,\alpha}$ boundary regularity, assuming standard structural conditions on the operator $F(D^2u, Du, x)$ and a form of H\"{o}lder regularity in the $x$-variable. Additionally, the author proved that solutions are $C^{2,\alpha}$ in a neighborhood of the boundary, without needing convexity assumptions, unlike the interior estimates as discussed above. This does not contradict the counterexamples in \cite{Nad-Vla13} since the singularities occur at the center of the domain. Under similar assumptions, the authors in \cite{Milakes-Silvestre2006} proved $C^{1,\alpha}$ and $C^{2,\alpha}$ regularity for a problem similar to \eqref{eq_main4}, with Neumann boundary conditions and $f \in L^p$, for $p > d$. Again, the $C^{2,\alpha}$-estimates depended on the convexity/concavity of the operator. In Sobolev spaces, boundary regularity results were also obtained in \cite{nikki2009}, where the author proved $W^{1,p}$ and $W^{2,p}$ boundary estimates for solutions to \eqref{1} with $f \in L^p$, $p \in (d - \varepsilon_0, +\infty)$, Dirichlet boundary conditions, and a $C^{1,1}$ domain. The strategy followed closely that of \cite{Caffarelli1989}. Once more the $W^{2, p}$-estimates rely on the convexity/concavity of the operator (or $C^{1, 1}$-estimates for the corresponding homogeneous problem).

The purpose of the present paper is to establish fractional Sobolev regularity up to the boundary for viscosity solutions to \eqref{1}, extending the results of \cite{Pimentel-Santos-Teixeira2022} to global estimates for gradient-dependent operators with variable coefficients. We adopt the strategy outlined in \cite{Pimentel-Santos-Teixeira2022}, simplifying it in some steps, by using $C^{1,\alpha}$-cones to touch the solution from above and below in order to derive certain estimates. Using this approach, we first prove that viscosity solutions to \eqref{eq_main3} with zero Dirichlet boundary conditions on a flat boundary are also weak solutions to a fractional Laplacian equation of order $\frac{1+\varepsilon}{2}$, with $\varepsilon \in (0,\alpha_0)$. As highlighted in \cite{Pimentel-Santos-Teixeira2022}, this suggests no better than fractional diffusion behaviour, even though $F$ is a second-order operator. In particular, we show that solutions are of class $W^{\gamma,p}$ up to the boundary, for all $\gamma < 1+\varepsilon$, with corresponding estimates. Unlike \cite{Silvestre-Sirakov2014} and \cite{nikki2009}, we assume only that the operator $F$ is uniformly elliptic (see Theorem \ref{theo1} and the discussion that follows) while deriving differentiability estimates of order strictly larger than one. Furthermore, we handle more general operators as in \eqref{1} with non-zero Dirichlet boundary conditions, where the boundary data is required to be of class $W^{2,p}$ due to the order of the operator. This is discussed in Theorem \ref{prop geral case II}. 

The remainder of the paper is structured as follows: In Section 2, we present definitions and state the main results. Section 3 is devoted to the proof of Theorem \ref{theo1} for equations with a flat boundary and zero boundary conditions. Section 4 extends Theorem \ref{theo1} to more general operators and $C^2$ domains. Finally, in the Appendix we establish local $W^{1+\varepsilon,p}$ for solutions to \eqref{1}, extending previous results in the literature for more general operators.


\section{Preliminaries}

\subsection{Definitions and auxiliary results}

This section gathers some definitions and auxiliary results used throughout this manuscript. We denote by $B_r(x_0)$ the ball of radius $r$ centered at $x_0$. In what follows, we give the definition of viscosity solutions.  

\begin{definition}[Viscosity solution]
A continuous function $u\in C(\Omega)$ is said to be a $L^p$-viscosity subsolution to \eqref{1} in $\Omega$, if for every $\varphi\in W^{2,p}_{\text{loc}}(\Omega)$ and $x_0\in \Omega$ such that $\varphi-u$ has a local maximum at $x_0$, there holds
\[
F(D^2u,Du,u,x) \leq f(x) \quad \text{a.e in}\quad \Omega',
\]
where $\Omega' \subset \Omega$ is a neighbourhood of $x_0$. Respectively, $u\in C(\Omega)$ is a $L^p$-viscosity supersolution to \eqref{1} in $\Omega$, for every $\varphi\in W^{2,p}_{\text{loc}}(\Omega)$ and $x_0\in \Omega$ such that $\varphi-u$ has a local minimum at $x_0$, there holds
\[
F(D^2u,Du,u,x) \geq f(x) \quad \text{a.e in}\quad \Omega',
\]
where $\Omega' \subset \Omega$ is a neighbourhood of $x_0$. We say that $u\in C(\Omega)$ is a $L^p$-viscosity solution to the equation \eqref{1} if it is simultaneously a viscosity subsolution and supersolution.
\end{definition}

For simplicity, we only refer to the solutions as viscosity solutions. As previously mentioned, we adopt the strategy outlined in \cite{Pimentel-Santos-Teixeira2022}, which utilizes cones to touch the solutions from both above and below to derive certain estimates. This contrasts with the approach in \cite{Caffarelli1989}, which employs paraboloids and requires the assumption of $C^{1,1}$-estimates. We will now define the $C^{1, \alpha}$-cones.

\begin{definition}[$C^{1, \alpha}$-cones]\label{def_cones}
We say that a function $\psi$ is a convex $C^{1,\alpha}$-cone of opening $M$ and vertex $x_0$ if 
\[
\psi(x)=\ell(x)+\frac{M}{2}|x-x_0|^{1+\alpha},
\]
where $M$ is a positive constant and $\ell$ is an affine function. Similarly, we say that a function $\psi$ is a concave $C^{1,\alpha}$-cone of opening $M$ and vertex $x_0$ if 
\[
\psi(x) = \ell(x)-\frac{M}{2}|x-x_0|^{1+\alpha},
\]
where $M$ is a positive constant and $\ell$ is an affine function.
\end{definition}

The set of points that can be touched by a $C^{1,\alpha}$-cone of a given opening $M$ is crucial to our analysis. More precisely, our strategy involves establishing a decay rate for these sets. 

\begin{definition}
Let $\Omega\subset\mathbb{R}^d$ be a bounded open set. For $M>0$ and $0<\tau_0<diam(\Omega')/5$, where $\Omega'\subset \Omega$ is an open subset, we define 
\[
\underline{G}_M(u,\Omega')=\underline{G}_M(\Omega')
\]
as the set of all points $x_0\in \Omega'$ for which there exists a convex $C^{1,\alpha}$-cone $\psi$ of opening $M$ such that
\begin{enumerate}
    \item[1.] $u(x_0) = \psi(x_0)$
    \item[2.] $u(x) < \psi(x)$ for all $x \in B_{\tau_0}(x_0) \setminus \{x_0\}$.
\end{enumerate}
Similarly, we define
\[
\overline{G}_M(u,\Omega')=\overline{G}_M(\Omega')
\]
as the set of all points $x_0\in \Omega'$ for which there exists a concave $C^{1,\alpha}$-cone $\psi$ of opening $M$ such that 
\begin{enumerate}
    \item[1.] $u(x_0)=\psi(x_0)$
    \item[2.] $u(x)>\psi(x)$ for all $x \in B_{\tau_0}(x_0)\setminus\{x_0\}$.
\end{enumerate}
Finally,
\[
G_M(\Omega'):=\overline{G}_M(\Omega')\cap\underline{G}_M(\Omega').
\]
\end{definition}

The following lemma gives some properties related to the set $G_M$.

\begin{lemma}\label{GM properties}
Given a bounded open set $\Omega$, we have that the sets $G_M$ enjoy the following properties:
\begin{enumerate}
    \item[i)] Let $M>N>0$, then $G_N(u,\Omega)\subset G_M(u,\Omega)$.\\
    \item[ii)] Let $K\neq 0$ be a constant, then $G_M(Ku,\Omega)=G_{M/K}(u,\Omega)$.\\
\end{enumerate}
\end{lemma}
\begin{proof}
The proof of the first item can be found in the discussion after \cite[Definition 4]{Pimentel-Santos-Teixeira2022}. For the second item, we only show that it holds for $\underline{G}_M(u,\Omega)$. Indeed, let $z\in\underline{G}_M(Ku,\Omega)$ then there exist a 
\[
\varphi(x)=\ell(x)-\frac{M}{2}|x-z|^{1+\alpha}
\]
such that $\varphi$ touch $Ku$ from above at $z$, i.e., $\varphi(z)=Ku(z)$ and $\varphi(x)>Ku(x)$, for $x \in B_{r}(z)\setminus\{z\}$. Observe that by defining 
\[
\tilde{\varphi}(y):=\frac{1}{K}\varphi(y)
\]
we have that $\tilde{\varphi}$ touch $u$ from above at $z$, which implies that $z\in\underline{G}_{M/K}(u,\Omega)$. Hence $\underline{G}_{M}(Ku, \Omega) \subset \underline{G}_{M/K}(u, \Omega)$ Analogously we can show that $\underline{G}_{M/K}(u,\Omega)\subset\underline{G}_M(Ku,\Omega)$. This finishes the proof.
\end{proof}

The complementary set of $G_M(u,\Omega)$, i.e., the set of points where a $C^{1,\alpha}$-cone cannot touch the graph of the solution, is also essential in our analysis:

\begin{definition}
Let $\Omega\subset\mathbb{R}^d$ be a bounded open set. For $M>0$ and $0<\tau_0<diam(\Omega')/5$, where $\Omega'\subset \Omega$ is an open subset, we define 
\[
\underline{A}_M(u,\Omega')=\underline{A}_M(\Omega')=\Omega'\setminus \underline{G}_M(u,\Omega').
\]
In addition, we set
\[
\overline{A}_M(u,\Omega')=\overline{A}_M(\Omega')=\Omega'\setminus \overline{G}_M(u,\Omega').
\]
Finally,
\[
A_M(u,\Omega'):=A_M(\Omega')=\Omega'\setminus G_M(u,\Omega').
\]
\end{definition}

We will now define the $C^{1,\alpha}$-aperture function, which is directly related to the integrability of our solutions.

\begin{definition}
For $x\in B\Subset B_1$ we define
\[
\theta(x):=\theta_{1+\alpha}(u,B)(x)=\inf\{M; x\in G_M(u,B)\}\in[0,\infty].
\]
\end{definition}

Next, we recall the definition of a distribution function and the result that connects it to the integrability of $\theta_{1+\alpha}$.

\begin{definition}
Let $\Omega\subset\mathbb{R}^d$, and $g:\Omega\longrightarrow\mathbb{R}$ be a nonnegative and measurable function. Define $\mu_g:\mathbb{R}^+_0\longrightarrow\mathbb{R}^+_0$ as
\[
\mu_g(t):=|\{x\in \Omega ;  g(x)>t\}|,
\]
for $t>0$. The function $\mu_g$ is called the distribution function of $g$.    
\end{definition}

\begin{lemma}\label{lem_eqsob}
 Let $\eta>0$ and $M>1$ be constants. Then, for $0<p<\infty$,
\[
g\in L^p(\Omega) \iff \sum_{k\geq 1}M^{pk}\mu_g(\eta M^k)=S<\infty
\]
and
\[
C^{-1}S\leq \|g\|^p_{L^p(\Omega)}\leq C(|\Omega|+S),
\]
where $C>0$ is a constant depending only on $\eta, M$ and $p$.
\end{lemma}

Notice that the distribution function of $\theta$ is related to the measure of the sets $A_M$. More precisely, we have 
\begin{equation}\label{eq_integtheta}
\mu_{\theta}(t) \leq |A_t(u, B_{1/2})|.    
\end{equation}
Therefore, in order to examine how information is transmitted from $\theta$ to $u$, it is essential to study the summability of 
\[
\sum_{k \geq 1} M^{pk} |A_{M^k}(u, B_{1/2})|.
\] 

In what follows, we will recall a consequence of the Calderón-Zygmund decomposition. Consider the unit cube $Q_1$, which is divided into $2^d$ subcubes, each with half the side length. Each of these subcubes is further subdivided in the same manner, and this process is repeated iteratively. The resulting cubes are known as dyadic cubes. If $Q$ is a dyadic cube distinct from $Q_1$, we refer to $\tilde{Q}$ as the predecessor of $Q$ if $Q$ is one of the $2^d$ subcubes formed by subdividing $\tilde{Q}$.

\begin{lemma}\label{lem_calzyg}
Let $A\subset B\subset Q_1$ be a measurable sets and take $|A|\leq\delta<1$ such that if $Q$ is a dyadic cube satisfying $|A\cap Q|>\delta|Q|$, so $\tilde{Q}\subset B$, where $\tilde{Q}$ is the predecessor of $Q$. Then 
\[
|A|\leq\delta |B|.
\]
\end{lemma}

\begin{proposition}\label{Maximal function}
Let $g\in L^1_{\text{loc}}(\mathbb{R}^d)$, the Maximal function of $g$, denoted by $M(g)$ is defined as
\[
M(g)(x):=\sup_{r>0}\frac{1}{Q_r(x)}\int_{Q_r(x)}|g(y)|dy.
\]
The Maximal operator satisfies
\[
|\{x\in \mathbb{R}^d; M(g)(x)\geq t\}|\leq\frac{C}{t}\|g\|_{L^1(\mathbb{R}^d)}, \quad \forall t>0,
\]
and for $g\in L^p(\mathbb{R}^d)$ with $1<p\leq\infty$ we have
\[
\|M(g)\|_{L^p(\mathbb{R}^d)}\leq C(d,p)\|g\|_{L^p(\mathbb{R}^d)}.
\]
Moreover, if $g\in L^p(\mathbb{R}^d)$ with $1\leq p\leq\infty$, then $M(g)$ is finite almost everywhere.
\end{proposition}

We now define a function $\beta$ that quantifies the oscillation of $F$ with respect to the variable $x$. This function is closely linked to the approximation techniques used in \cite{Caffarelli1989} and subsequent works. We refer to Assumption $A\ref{A-oscF}$ below for further details.

\begin{definition}\label{def_oscF}
Let $F:\mathcal{S}(d)\times\Omega\rightarrow\mathbb{R}^d$ be continuous in $x$. We define
\[
\beta_{F}(x,y):=\beta(x,y)=\sup_{M\in S(d)\setminus\{0\}}\frac{|F(M,x)-F(M,y)|}{\|M\|}.
\]
\end{definition}

In the sequence, we define the fractional Sobolev spaces. We refer the reader to \cite[chapter 2]{Nezza-Palatucci-Valdinoci-2012}

\begin{definition}
Let $s\in (0,1)$ fixed. For any $p\in [1,\infty)$ we define the space $W^{s,p}(\Omega)$ as follows
\[
W^{s,p}(\Omega):=\left\{u \in L^p(\Omega); \frac{|u(x)-u(y)|}{|x-y|^{\frac{d}{p}+s}}\in L^p(\Omega\times\Omega)\right\},
\]
equipped with a natural norm
\[
\|u\|_{W^{s,p}(\Omega)}:=\left(\int_{\Omega}|u|^pdx+\int_{\Omega}\int_{\Omega}\frac{|u(x)-u(y)|^p}{|x-y|^{d+sp}}dxdy\right)^{1/p},
\]
where
\[
[u]_{W^{s,p}(\Omega)}:=\left(\int_{\Omega}\int_{\Omega}\frac{|u(x)-u(y)|^p}{|x-y|^{d+sp}}dxdy\right)^{1/p}
\]
is the so-called Gagliardo semi-norm of $u$.
\end{definition}

In the same spirit, the Sobolev spaces involving weak derivatives of higher fractional order are defined in the following way:
\begin{definition}
If $s>1$, we write $s=m+\gamma$, where $m$ is an integer and $\gamma\in(0,1)$. We define
\[
W^{s,p}(\Omega):=\{u \in W^{m,p}(\Omega);D^{\alpha}u\in W^{\gamma,p}(\Omega),\text{ for any }\alpha \text{such that }  |\alpha|=m\},
\]
equipped with the following norm
\[
\|u\|_{W^{s,p}(\Omega)}:=\left(\|u\|^p_{W^{m,p}(\Omega)}+\sum_{|\alpha|=m}\|D^{\alpha}u\|^p_{W^{\gamma,p}(\Omega)}\right)^{1/p}.
\]
\end{definition}

Finally, we say that a given constant is universal when it only depends on the dimension $d$ and on the ellipticity constants $\lambda$ and $\Lambda$.

\subsection{Assumptions and main results}

Next, we detail the main assumptions of the paper. We begin with the uniformly elliptic condition on $F$.

\begin{Assumption}[Uniformly elliptic]\label{A-ellip}
We suppose that $F : {\mathcal S}(d)\times B_1^+ \to \R  $ is uniformly elliptic, i.e., there exist positive constants $\lambda \leq \Lambda$ such that 
\[
\lambda\|N\| \leq F(M+N,x) - F(M,x) \leq \Lambda\|N\|,
\]
for any $x \in B_1$ and $M, N \in {\mathcal S}(d)$, with $N \geq 0$.
\end{Assumption}

The next assumption concerns the integrability of the source term.

\begin{Assumption}[Integrability o the source term]\label{A-sourceterm} We assume that $f \in L^p(B_1^+)$, for $p > d$. Moreover, there exists a positive constant $C$ such that
\[
\|f\|_{L^p(B_1^+)} \leq C.
\]
\end{Assumption}

To establish the first result (see Theorem \ref{theo1} below), we need to assume the continuity of $F$ in the variable $x$, allowing us to connect our operator to another one with constant coefficients.

\begin{Assumption}[Continuity of $F$]\label{A-cont}
We assume the operator $F: {\mathcal S}(d)\times B_1^+ \to \R $ is continuous in the variable $x$.    
\end{Assumption}

As mentioned above, a smallness regime on the oscillation of $F$ in the $x$ variable is essential to apply the approximation procedure presented in \cite{Caffarelli1989}. This is the content of the next assumption.

\begin{Assumption}[Oscillation decay]\label{A-oscF} We assume that there exists a sufficiently small $\beta_0 > 0$ such that
\[
\dashint_{B_{r}(x_0)\cap B_1^+}\beta_{F}(x, x_0)^ddx \leq \beta_0^d,
\]
for any $x_0 \in B_1^+$ and all $r > 0$.
\end{Assumption}

We now present the main results of this paper, starting with a result concerning zero Dirichlet boundary conditions on a flat boundary.

\begin{theorem}\label{theo1}
Let $u \in C(\overline{B}^+_1)$ be a bounded viscosity solution to
\begin{equation*}
\begin{cases}
    F(D^2u,x)=f(x) \quad &\text{in} \quad B^+_1\\
    u=0 \quad &\text{on} \quad \partial B^+_1.
\end{cases}
\end{equation*}
Assume that $A\ref{A-ellip}$-$A\ref{A-oscF}$ are satisfied and $p \in (d,+\infty)$ and $\varepsilon\in (0,\alpha_0)$. Then we can find $g \in L^p(B^+_{1/2})$ such that the function $\tilde{u} := u\chi_{B_{1/2}^+}$ is a weak solution to
\begin{equation}
\begin{cases}\label{eq_fraclap1}
    (-\Delta)^{\frac{\sigma}{2}}\tilde{u}=g \quad &\text{in} \quad B^+_{1/2},\\
    \tilde{u}=0 \quad &\text{in} \quad \R^d\setminus B^+_{1/2},
\end{cases}
\end{equation}
where $\sigma \in (0, 1+\varepsilon)$.
\end{theorem}

We follow the approach presented in \cite{Pimentel-Santos-Teixeira2022}, with several improvements. Using Lemma \ref{lem_eqsob}, we show that $\theta_{1+\alpha} \in L^p(B^+_{1/2})$, which enables us to prove that a fractional-order quotient of $u$ also lies in $L^p$. After further analysis, we demonstrate that $u$ satisfies \eqref{eq_fraclap1}. This method contrasts significantly with \cite{Caffarelli1989}, where paraboloids, rather than cones, are employed, allowing for the straightforward conclusion that $\theta \in L^p$ implies $D^2u \in L^p$, thus yielding the desired $W^{2,p}$-regularity. Additionally, Sobolev embeddings ensure that Theorem \ref{theo1} is optimal in terms of differentiability; otherwise, we would conclude that the solutions belong to $C^{1,\alpha}(B_1^+)$ for $\alpha > \alpha_0$, leading to a contradiction.

By employing the regularity theory known for the fractional Laplacian, we can obtain the following result:

\begin{corollary}\label{cor_fracreg}
Under the hypotheses of Theorem \ref{theo1}, we have that $u \in W^{\gamma, p}(B_{1/2}^+)$ with the estimate
\[
\|u\|_{W^{\gamma,p}(B_{1/2}^+)} \leq C(\|u\|_{L^\infty(B_1^+)} +\|f\|_{L^p(B_1^+)}),
\]
for all $\gamma < \sigma$, where $C=C(d,\lambda,\Lambda,\alpha_0,\sigma,p)$ is a positive constant.
\end{corollary}

To establish similar regularity for operators as in \eqref{1}, we require a condition equivalent to Assumption $A\ref{A-ellip}$. To this end, we introduce the following (nowadays standard) structural condition.

\begin{Assumption}[Structure condition on $F$]\label{A-strcond}
We suppose that $F$ satisfies the inequality
{\small
\[
{\mathcal M}_{\lambda, \Lambda}^-( M-N) -b|p-q| -c|r-s| \leq F(M, p, r, x) - F(N, q, s, x) \leq {\mathcal M}^+_{\lambda, \Lambda}( M-N)+b|p-q| +c|r-s|,
\]
}
for all $M, N \in {\mathcal S}(d)$, $p, q \in \R^d$, $r, s \in \R$, $x \in \Omega$, and constants $b,c > 0$. Observe that, if $b\equiv c\equiv 0$, then we recover the assumption A\ref{A-ellip}.
\end{Assumption}

We are now ready to state our second main result.

\begin{theorem}\label{prop geral case II}
Let $u\in C(\overline{\Omega})$ be a bounded viscosity solution to
\begin{equation}\label{eq sol prop geral case II}
    \begin{cases}
        F(D^2u,Du,u,x)=f(x) \quad &\text{in} \quad \Omega,\\
        u=\varphi \quad & \text{on} \quad \partial\Omega,
    \end{cases}
\end{equation}
where $\Omega\Subset\mathbb{R}^d$, with $\partial\Omega\in C^{1,1}$, $\varphi\in W^{2,p}(\overline{\Omega})\cap C(\overline{\Omega})$. Suppose further that A\ref{A-sourceterm}, A\ref{A-oscF} and A\ref{A-strcond} hold and $\varepsilon \in (0, \alpha_0)$. Then there exists $\varepsilon_0= \varepsilon_0(d,\lambda/\Lambda,b)$ and $C=C(d,\lambda,\Lambda,b,c,p,r_0)$, such that if $p \in (d-\varepsilon_0,+\infty)$, then $u\in W^{\gamma,p}(\Omega)$ and
\begin{equation}\label{eq prop geral case II}
\|u\|_{W^{\gamma,p}(\Omega)}\leq C\left(\|u\|_{L^{\infty}(\Omega)}+\|f\|_{L^p(\Omega)}+\|\varphi\|_{W^{2,p}(\overline{\Omega})}\right),
\end{equation}
for all $\gamma < 1+\varepsilon$, where $C=C(d,\lambda,\Lambda,\alpha_0,\sigma,p,\overline{\Omega})$ is a positive constant.
\end{theorem}

Notice that for non-zero boundary conditions, the boundary data must lie in $W^{2,p}$ due to the second-order nature of the operator. The proof of Theorem \ref{prop geral case II} builds on Theorem \ref{theo1}, Corollary \ref{cor_fracreg} and the arguments from \cite{nikki2009}, including the already established $W^{1,p}$ regularity.

\begin{remark}\label{rem_escau}
Notice that as in \cite{nikki2009}, the condition $p \in (d, +\infty)$, can be improved to $p \in (d-\varepsilon_0, +\infty)$. It is sufficient to use the results from \cite{Escauriaza93}, together with the boundary Harnack inequality and global H\"{o}lder inequality for $W^{2, d-\varepsilon_0}$ viscosity solutions from \cite{nikki2009}.     
\end{remark}

\begin{remark}\label{rem_w2p}
We assume that viscosity solutions to
\[
F(D^2u, x) = 0 \quad \mbox{in} \quad B^+_1,
\]
belong to $W^{2,p}(B_1^+)$, though estimates in these spaces are not utilized. This result is obtained through an approximation procedure, detailed in Proposition \ref{prop_w2papp} in the Appendix (Section \ref{appendix}). 
\end{remark}

\begin{remark}[Scaling properties]\label{rem_scaling}
Throughout the paper, we assume smallness conditions on certain quantities, as the norms of u and the source term f. We want to stress that such conditions are not restrictive. If $u$ is a viscosity solution to \eqref{1}, then for $\varepsilon > 0$ the normalized function
\[
w(x) := \dfrac{u(x)}{\|u\|_{L^\infty(B_1^+)} + \varepsilon^{-1}\|f\|_{L^p(B_1^+)}}
\]
is such that
$\|w\|_{L^\infty(B_1^+)} \leq 1$, and solves
\[
\tilde{F}(D^2w, x) := \tilde{f}(x) \quad \mbox{in} \quad B^+_1,
\]
where 
\[
\tilde{F}(M, x) = \dfrac{1}{\kappa}F(\kappa M, x)
\]
has the same ellipticity constants as $F$, and
\[
\tilde{f}(x) := \dfrac{f(x)}{\kappa},
\]
for $\kappa := \|u\|_{L^\infty(B_1^+)} + \varepsilon^{-1}\|f\|_{L^p(B_1^+)}$. In particular, we have $\|\tilde{f}\|_{L^p(B_1)^+} \leq \varepsilon$.
\end{remark}

\section{Estimates at the boundary}

In this section, given $\varepsilon \in (0, \alpha_0)$, we prove boundary $W^{\gamma,p}$-estimates for solutions to
\begin{equation}
\begin{cases}\label{eq_mainonlyx}
F(D^2u, x) = f(x) & \mbox{ in } \; B_1^+   \\
u = 0 & \mbox{ on } \; B_1^+\cap \{x_d=0\},
\end{cases}  
\end{equation}
for all $\gamma < 1+\varepsilon$. We begin with a rescaled version of the interior decay rate for the sets $A_t(u, \Omega)$, which can be found in \cite[Proposition 4]{Pimentel-Santos-Teixeira2022}.   

\begin{lemma}\label{Lemma 1}
Let $u\in C(B_{6r\sqrt{d}})$ be a viscosity supersolution to \eqref{1} in $B_{6r\sqrt{d}}$. Suppose $A\ref{A-ellip}$ and $A\ref{A-sourceterm}$ are in force and $\Omega$ is a bounded domain. Extend $f$ by zero outside $B_{6r\sqrt{d}}$. Then there are universal positive constants $\mu,\delta_0$ and $C$, such that $\|f\|_{L^d(B_{6r\sqrt{d}}(x_0))}\leq r^{-1}\delta_0$ implies

\[
|\underline{A}_{t}(u,\Omega)\cap Q_{r}|\leq Ct^{-\mu}r^{-2\mu}|Q_{r}| \hspace{0.4cm} \forall t>0.
\]
\end{lemma}

We emphasize that if $u$ is a viscosity solution, the result holds for $A_t(u, \Omega)$. 

\begin{lemma}\label{lemma2}
Let $u\in C(\Omega)$ be a normalized viscosity supersolution to \eqref{1} in a bounded domain $\Omega$, assume that $A\ref{A-ellip}$ and $A\ref{A-sourceterm}$ are in force. Then for any $\Omega'\Subset\Omega$ we have

\[
|\underline{A}_{t}(u,\Omega)\cap \Omega'|\leq C(d,\lambda,\Lambda,\Omega,dist(\Omega',\partial\Omega))(1+\|f\|_{L^p(\Omega)})^{\mu}t^{-\mu}
\]
\end{lemma}

\begin{proof}
Fix $\varepsilon\in (0,1)$ such that $6\varepsilon\sqrt{d}<dist(\Omega',\partial\Omega)$ and choose a finite cover of $\Omega'$ with parallel cubes having side-length $\varepsilon$ and disjoint interior. Observe that $B_{6\varepsilon\sqrt{d}}(x_i)\Subset\Omega$ for all centers $x_i$ of cubes $Q_{\varepsilon}(x_i)$. Let us define the auxiliary function
\[
\tilde{u}:=\frac{u}{\varepsilon\sqrt{d}\delta_0^{-1}\|f\|_{L^d(\Omega)}},
\]
so $\tilde{u}$ solves a similar equation as in \eqref{1} with 
\[
\tilde{f}:=\frac{f}{\varepsilon\sqrt{d}\delta_0^{-1}\|f\|_{L^d(\Omega)}},
\]
which satisfies $\|\tilde{f}\|_{L^d} \leq r^{-1}\delta_0$
and $\tilde{f}$ satisfies the hypothesis of the previous lemma in $B_{6\varepsilon\sqrt{d}}(x_i)$, and as a consequence we get
\[
|\underline{A}_{t}(\tilde{u},\Omega)\cap Q_{\varepsilon}(x_i)|\leq Ct^{-\mu}\varepsilon^{-2\mu}|Q_{\varepsilon}(x_i)|=C\varepsilon^dt^{-\mu}\varepsilon^{-2\mu}.
\]
Now, denoting $N=N(\Omega',\varepsilon)$ as the number of cubes that covers $\Omega'$ we have for $K = \dfrac{1}{\varepsilon\sqrt{d}\delta_0^{-1}\|f\|_{L^d(\Omega)}}$
\begin{align*}
|\underline{A}_{t/K}({u},\Omega)\cap \Omega'| & = |\underline{A}_{t}(\tilde{u},\Omega)\cap \Omega'|\\
&\leq \left|\bigcup_{i=1}^{N} \underline{A}_{t}(\tilde{u},\Omega)\cap Q_{\varepsilon}(x_i)\right|\\
&\leq \sum_{i=1}^{N}Ct^{-\mu}\varepsilon^{-2\mu}|Q_{\varepsilon}(x_i)|=CN\varepsilon^dt^{-\mu}\varepsilon^{-2\mu}\\
&\leq
Ct^{-\mu}\varepsilon^{-2\mu}, 
\end{align*}
where in the first equality we have used Lemma \ref{GM properties} ii) and in the last inequality we have used that $N\varepsilon^d\leq C(d)|\Omega|$. This finishes the proof.
\end{proof}

The next lemma gives our first estimate of $|G_t(u, \Omega)|$. In what follows $\Omega$ is a bounded domain such that $B^+_{14\sqrt{d}} \subset \Omega \subset \R^d_+$.

\begin{lemma} \label{lemma3}
Let $u\in C(\Omega)$ be a viscosity solution to \eqref{1} in $B^{+}_{12\sqrt{d}}$ and suppose that A\ref{A-ellip} and A\ref{A-sourceterm} hold true. There exist universal constants $t>1$ and $\sigma\in (0,1)$ such that if $\|f\|_{L^p(B^{+}_{12\sqrt{d}})}\leq 1$, then
\[
|\underline{G}_t(u,\Omega)\cap\left((Q_1^{d-1}\times(0,1))+x_0\right)|\geq 1-\sigma,
\]
for any $x_0\in B_{9\sqrt{d}}\cap\{x_d\geq 0\}$.
\end{lemma}

\begin{proof}
Fix $0 < \sigma < 1$ and denote $x_0 = (x_0', x_{0,d})$. We split the proof into two cases. First, consider the case where $x_{0,d} \geq \sigma/2$, in this case, we can apply directly Lemma \ref{lemma2} to obtain the result. The complementary case where $x_{0,d} < \sigma/2$, we observe that $(Q_1^{d-1} \times (0,1)) + x_0$
is a subset of 
\[
 \left( \left(Q_1^{d-1} \times (x_{0,d}, \frac{\sigma}{2}]\right) + (x_0', 0) \right) 
\cup 
\left( \left(Q_1^{d-1} \times (\frac{\sigma}{2}, x_{0,d} + 1)\right) + (x_0', 0) \right).
\]
Now, using this fact and applying Lemma \ref{lemma2} again, we deduce that 
\begin{align*}
&\left|\underline{A}_t(u, \Omega) \cap \left( (Q_1^{d-1} \times (0,1)) + x_0 \right)\right|\\
&\leq \left| \left( \left(Q_1^{d-1} \times (x_{0,d}, \frac{\sigma}{2}]\right) + (x_0', 0) \right) \right| 
+ \left|\underline{A}_t(u, \Omega) \cap \left( (Q_1^{d-1} \times (\frac{\sigma}{2}, x_{0,d} + 1)) + (x_0', 0) \right)\right| \\
&\leq \frac{\sigma}{2}+C(d, \Omega, \lambda, \Lambda, \sigma)t^{-\mu}\\
&\leq \sigma,
\end{align*}
where the last inequality is obtained  choosing $t \geq (2C(d, \Omega, \lambda, \Lambda, \sigma)/\sigma)^{\frac{1}{\mu}}$. To finish the proof, notice that
\[
\left|\left(Q_1^{d-1} \times (0,1)\right) + x_0\right|\leq \left|\underline{G}_{t}(u,\Omega)\cap\left(\left(Q_1^{d-1} \times (0,1)\right) + x_0\right)\right|+\left|\underline{A}_{t}(u,\Omega)\cap\left(\left(Q_1^{d-1} \times (0,1)\right) + x_0\right)\right|
\]
which implies the desired result.
\end{proof}

\begin{lemma}\label{lemma4}
Let $u\in C(\Omega)$ be a viscosity solution to \eqref{1} in $B_{12\sqrt{d}}^+$ and suppose that A\ref{A-ellip} and A\ref{A-sourceterm} are in force. Assume further that
\[
G_1(u,\Omega)\cap\left(\left(Q_2^{d-1}\times(0,2)\right)+\tilde{x}_0\right)\neq \emptyset
\]
for some $\Tilde{x}_0\in B_{9\sqrt{d}}\cap\{x_d\geq 0\}$. There exist universal constants $M>1$ and $0<\sigma<1$, such that if
\[
\|f\|_{L^d(B_{12\sqrt{d}}^+)}\leq 1, 
\]
then, for any $x_0\in B_{9\sqrt{d}}\cap\{x_d \geq 0\}$
\[
\left|G_M(u,\Omega)\cap\left(\left(Q_1^{d-1}\times(0,1)\right)+x_0\right)\right|\geq 1-\sigma.
\]
\end{lemma}

\begin{proof}
Let $x_1 \in G_1(u, \Omega) \cap \left( (Q_2^{d-1} \times (0,2)) + \tilde{x}_0 \right)$, then there exist $C^{1,\alpha}$-cones such that
\[
\ell(x) - \frac{1}{2} |x - x_1|^{1+\alpha} \leq u(x) \leq L(x) + \frac{1}{2} |x - x_1|^{1+\alpha}
\]
for all $x \in \Omega$, where $\ell$ is an affine function. Now, define
\[
v(x) := \frac{u(x) - L(x)}{C},
\]
where $C := C(\alpha, d)$ is chosen sufficiently large such that $\|v\|_{L^{\infty}(B_{12\sqrt{d}}^+)} \leq 1$ and

\[
- |x|^{1+\alpha} \leq v(x) \leq |x|^{1+\alpha} \quad \text{in } \Omega \setminus B_{12\sqrt{d}}^+.
\]

Notice that $v$ is a viscosity solution to
\[
\tilde{F}(D^2 v, x) = \tilde{f}(x) \quad \text{in } B_{12\sqrt{d}}^+,
\]
where
\[
\tilde{F}(D^2 v, x) := \frac{1}{C} F(C D^2 v, x), \quad \text{and} \quad \tilde{f}(x) := \frac{1}{C} f(x).
\]
Now, by applying Lemma \ref{lemma3} to $v$, we guarantee the existence of universal constants $M > 1$ and $0 < \sigma < 1$ such that
\[
\left| G_M(v, B_{12\sqrt{d}}^+) \cap \left( \left(Q_1^{d-1} \times (0,1)\right) + x_0 \right) \right| \geq 1 - \sigma
\]
for any $x_0 \in B_{9\sqrt{d}} \cap \{ x_d \geq 0 \}$. 

Using these estimates and the monotonicity property of the sets $G_t$, we conclude that

\[
\left| G_N(v, \Omega) \cap \left( \left(Q_1^{d-1} \times (0,1)\right) + x_0 \right) \right| \geq \left| G_M(v, B_{12\sqrt{d}}^+) \cap \left( \left(Q_1^{d-1} \times (0,1)\right) + x_0 \right) \right| \geq 1 - \sigma,
\]
for sufficiently large $N \geq M$. We finish the proof by observing that from the definition of $v$ and Lemma \ref{GM properties} we have
\[
G_M(v, \Omega) = G_{MC}(u, \Omega).
\]
\end{proof}

\begin{lemma}\label{lemma 5}
Let $u\in C(\Omega)$ be a normalized viscosity solution to \eqref{1} in $B_{12\sqrt{d}}^+$, and suppose that A\ref{A-ellip} and A\ref{A-sourceterm} hold with $\|f\|_{L^{p}(B_{12\sqrt{d}}^+)}\leq 1$. Extend $f$ by zero outside $B_{12\sqrt{d}}^+$ and define
\begin{align*}
A&:=A_{M^{k+1}}(u,\Omega)\cap\left(Q_1^{d-1}\times(0,1)\right)\\
B&:=\left(A_{M^{k}}(u,\Omega)\cap\left(Q_1^{d-1}\times(0,1)\right)\right)\cup\{x\in Q_1^{d-1}\times(0,1); M(f^d)\geq(C_0M^k)^d\}
\end{align*}
for $k\in\mathbb{N}_0$. Then $|A|\leq\sigma|B|$, where $C_0=C_0(d)$, $M>1$ and $\sigma\in(0,1)$ are universal constants.
\end{lemma}

\begin{proof}
First, observe that since $M^{k+1} \geq M^k$ for $M > 1$ and $k \in \mathbb{N}_0$, we have the inclusion $A_{M^{k+1}}(u, \Omega) \subset A_{M^k}(u, \Omega)$. As a result, $A \subset B$, and both are subsets of $\left(Q_1^{d-1} \times (0, 1)\right)$. Applying Lemma \ref{lemma3}, we obtain $|A| \leq \sigma < 1$. 

In order to use the Calderón-Zygmund cube decomposition argument, we need to show that for any dyadic cube $Q$ for which $|A \cap Q| > \sigma |Q|$, the set $\tilde{Q} \subset B$. Assume that for some $i \geq 1$, the cube 
\[
Q = \left(Q_{\frac{1}{2^i}}^{d-1} \times \left(0, \frac{1}{2^i}\right)\right) + x_0
\]
is a dyadic cube with the predecessor 
\[
\tilde{Q} = \left(Q_{\frac{1}{2^{i-1}}}^{d-1} \times \left(0, \frac{1}{2^{i-1}}\right)\right) + \tilde{x}_0.
\] 
Suppose that $Q$ satisfies
\begin{equation} \label{AQ}
    |A \cap Q| = |A_{M^{k+1}}(u, \Omega) \cap Q| > \sigma |Q|,
\end{equation}
but $\tilde{Q} \nsubseteq B$. This implies that there exists a point $x_1 \in \tilde{Q} \setminus B, $ which means
\begin{equation} \label{x_1}
x_1 \in \tilde{Q} \cap G_{M^k}(u, \Omega) \quad \mbox{and}\quad M(f^d)(x_1) < (C_0 M^k)^d.
\end{equation}
Now, consider the transformation $T(y) := \tilde{x}_0 + 2^{-i}y$ and define the rescaled functions 
\[
\tilde{u}(y) := \frac{2^{2i}}{M^k}u(T(y)),
\]
and $\tilde{\Omega} = T^{-1}(\Omega)$. Since $i \geq 1$ and $\tilde{Q} \subset \left(Q_1^{d-1} \times (0,1)\right)$, we have the inclusion $B_{12\sqrt{d}/2^i}^+(\tilde{x}_0) \subset B_{12\sqrt{d}}^+$, which shows that $\tilde{u}$ is a viscosity solution of
\[
\tilde{F}(D^2 \tilde{u}) = \tilde{f} \quad \text{in} \quad B_{12\sqrt{d}}^+,
\]
where $\tilde{f}(y) := \frac{1}{M^k}f(T(y))$. Moreover, since $|\tilde{x}_0 - x_1|_{\infty} <2^{1-i}$, it follows that $B_{12\sqrt{d}/2^i}^+(\tilde{x}_0) \subset Q_{28\sqrt{d}/2^i}(x_1)$. Thus, from \eqref{x_1}, we deduce that
\begin{align*}
\|\tilde{f}\|^d_{L^d(B_{12\sqrt{d}}^+)} &=\frac{2^i{d}}{M^{kd}}\int_{B_{12\sqrt{d}}^+}|f(x)|^d dx\\
&\leq \frac{C(d)}{M^{kd}} \left(\int_{Q_{28\sqrt{d}/2^i}(x_1)} |f(x)|^d dx\right)^{\frac{1}{d}}\\
&\leq C(d) C_0\\
&\leq 1
\end{align*}
for a sufficiently small constant $C_0 > 0$. Furthermore, from \eqref{x_1}, we can also conclude that
\[
G_1(\tilde{u}, \tilde{\Omega}) \cap \left(Q_2^{d-1} \times (0, 2)\right) \neq \emptyset,
\]
which shows that the conditions of Lemma \ref{lemma4} hold in $\tilde{\Omega}$. 

Since $x_{0,n} \geq \tilde{x}_{0,n}$ and $|x_0 - \tilde{x}_0| \leq \sqrt{d}/2^i$, it follows that $2^i(x_0 - \tilde{x}_0) \in B_{9\sqrt{d}} \cap \{x_d \geq 0\}$. Applying Lemma \ref{lemma4}, we get
\[
\left| G_M(\tilde{u}, \tilde{\Omega}) \cap \left(\left(Q_1^{d-1} \times (0,1)\right) + 2^i(x_0 - \tilde{x}_0)\right)\right| \geq 1 - \sigma,
\]
which implies
\[
\left| G_{M^{k+1}}(u, \Omega) \cap Q \right| \geq (1 - \sigma) |Q|.
\]

This contradicts \eqref{AQ}, completing the proof.
\end{proof}

From Lemma \ref{lemma 5} we derive a power decay for $|A_t(u,\Omega)|$ at the boundary.

\begin{lemma}\label{lemma 6}
Let $u\in C(\Omega)$ be a normalized viscosity solution in $B_{12\sqrt{d}}^+$, and suppose that A\ref{A-ellip} and A\ref{A-sourceterm} are in force. There exist universal constants $C>0$ and $\mu>0$ such that if $\|f\|_{L^d(B^+_{12\sqrt{d}})} \leq 1$, then
\[
\left|A_t(u,\Omega)\cap\left(\left(Q_1^{d-1}\times(0,1)\right)+x_0\right)\right|\leq Ct^{-\mu}
\]
for any $x_0\in B_{9\sqrt{d}}\cap\{x_d\geq 0\}$.
\end{lemma}

\begin{proof}
We will split the proof into two parts. Suppose first that  $x_0 = 0$ and define
\[
\alpha_k := \left| A_{M^k}(u, \Omega) \cap \left(Q_1^{d-1} \times (0, 1)\right) \right|, \quad
\beta_k := \left| \left\{x \in Q_1^{d-1} \times (0,1) : M(f^d)(x) \geq (C_0 M^k)^d \right\} \right|.
\]
Applying Lemma \ref{lemma 5}, we have the estimate
\[
\alpha_{k+1} \leq \sigma (\alpha_k + \beta_k).
\]
By iterating the inequality above we obtain
\[
\alpha_k \leq \sigma \alpha_{k-1} + \sigma \beta_{k-1} \leq \sigma^2 \alpha_{k-2} + \sigma^2 \beta_{k-2} + \sigma \beta_{k-1} \leq \dots \leq \sigma^k \alpha_0 + \sum_{j=0}^{k-1} \sigma^{k-j} \beta_j.
\]
Thus, we can conclude that
\[
\alpha_k \leq \sigma^k + \sum_{j=0}^{k-1} \sigma^{k-j} \beta_j.
\]
From Proposition \ref{Maximal function} we obtain
\[
\beta_j \leq C(d) (C_0 M^j)^{-d} \|f^d\|_{L^1(B^+_{12\sqrt{d}})} \leq C M^{-jd}.
\]
Therefore, we can estimate $\alpha_k$ as follows:
\[
\alpha_k \leq \sigma^k + C \sum_{j=0}^{k-1} \sigma^{k-j} M^{-dj} \leq \max(\sigma, M^{-d})^k + C \sum_{j=0}^{k-1} \max(\sigma, M^{-d})^{k-j} \max(\sigma, M^{-d})^j.
\]
Simplifying the expression above yields to
\[
\alpha_k \leq (1 + Ck) \max(\sigma, M^{-d})^k \leq CM^{-\mu k}.
\]
Finally, the first case is proved by choosing $\mu$ sufficiently small.

Now, suppose that $x_0 \neq 0$. Set
\[
\alpha_k := \left| A_k(u, \Omega) \cap \left(\left(Q_1^{d-1} \times (0, 1)\right) + x_0\right) \right|.
\]
Let $\varepsilon > 0$ be small enough so that $6 \varepsilon \sqrt{d} < \text{dist}\left( \left(Q_1^{d-1} \times (0, 1)\right) + x_0, \partial \Omega \right)$. We then can cover $\left(Q_1^{d-1} \times (0, 1)\right) + x_0$ using a finite collection of parallel cubes with side length $\varepsilon$, ensuring that these cubes are disjoint. Let $x_i$ denote the center of each cube $Q_{\varepsilon}(x_i)$. Applying Lemma \ref{Lemma 1} in the ball $B_{6\varepsilon\sqrt{d}}(x_i)$ for each center $x_i$, we deduce
\[
\alpha_k \leq \left|\bigcup_{i=1}^{N} A(u, \Omega) \cap Q_{\varepsilon}(x_i) \right| \leq C t^{-\mu} \varepsilon^{-2\mu}.
\]
This finishes the proof
\end{proof}

Now, combining the previous result with \eqref{eq_integtheta}, we obtain the first level of integrability for $\theta_{1+\alpha}$. 

\begin{corollary}[$L^\delta$-integrability of $\theta_{1+\alpha}$]
Let $u \in C(\overline{B}^+_1)$ be a normalized viscosity solution to
\begin{equation*}
\begin{cases}
    F(D^2u,x)=f(x) \quad &\text{in} \quad B^+_1\\
    u=0 \quad &\text{on} \quad \partial B^+_1.
\end{cases}
\end{equation*}
Suppose A\ref{A-ellip} and A\ref{A-sourceterm} are in force. Then $\theta_{1+\alpha} \in L^\delta(B_{1/2}^+)$, for $\delta \ll 1$ universal, and there exists a positive universal constant $C$ such that
\[
\int_{B^+_{1/2}}|\theta_{1+\alpha}(x)|^\delta dx \leq C.
\]
\end{corollary}

Our goal now is to accelerate the decay rate in Lemma \ref{lemma 6}, to improve the previous result. The next lemma is one of the main ingredients in such arguments.

\begin{proposition}[Approximation Lemma]\label{aproximationlemma}
Let $u\in C(\overline{B}_{14\sqrt{d}}^+)$ be a normalized viscosity solution to
\begin{equation*}
\begin{cases}
F(D^2u,x)=f(x)&\hspace{.2in}\mbox{in}\hspace{.2in}B_{14\sqrt{d}}^+\\
u=0&\hspace{.2in}\mbox{on}\hspace{.2in} B_{14\sqrt{d}}\cap\{x_d=0\}.
\end{cases}
\end{equation*}
Assume that A\ref{A-ellip}-A\ref{A-oscF} are in force. Given $\delta>0$, there exists $0<\varepsilon<\delta$ such that if 
\[
\|f\|_{L^d(B_{14\sqrt{d}}^+)}\leq \varepsilon,
\]
then we can find a function $h\in C^{1,\alpha_0}_{loc}(\overline{B}_{12\sqrt{d}}^+)$ satisfying
\[
\|u-h\|_{L^{\infty}(B_{12\sqrt{d}}^+)}\leq \delta.
\]
Moreover, there exists a universal constant $C=C(d,\lambda,\Lambda,\alpha_0) > 0$ such that
\[
\|h\|_{C^{1,\alpha_0}(\overline{B}_{12\sqrt{d}}^+)}\leq C.
\]
\end{proposition}

\begin{proof}
We will argue by contradiction. Suppose there exists $\delta_0>0$ and sequences $(F_n)_{n\in \mathbb{N}}$, $(u_n)_{n\in \mathbb{N}}$ and $(f_n)_{n\in \mathbb{N}}$, satisfying
\[
F_n \mbox{ is a } (\lambda,\Lambda)-\mbox{elliptic operator for each } n,
\]
\begin{equation}\label{7}
F_n(D^2u_n,x)=f_n(x) \hspace{0.2cm} \text{in}\hspace{0.2cm} B_{14\sqrt{d}}^+,
\end{equation}
\[
\|f\|_{L^d(B_{14\sqrt{d}}^+)}\leq\frac{1}{n}
\]
but
\begin{equation}\label{8}
    \|u-h\|_{L^{\infty}(B_{12\sqrt{d}}^+)}> \delta_0,
\end{equation}
for every $h \in  C^{1,\alpha_0}_{loc}(\overline{B}_{12\sqrt{d}}^+)$.

From the regularity available for \eqref{7}, we can find a function $u_{\infty}\in C_{loc}^{\frac{\beta}{2}}(B_{12\sqrt{d}}^+)$, for some $\beta \in (0,1)$, such that $u_n \to u_{\infty}$ locally uniformly in $B_{12\sqrt{d}}^+$. In addition, from the ellipticity of $F$ (assumption $A$\ref{A-ellip}), there exists a $(\lambda,\Lambda)$-elliptic operator $F_{\infty}$ such that $F_n \to F_{\infty}$. Now, stability results imply that $u_{\infty}$ solves
\[
F_{\infty}(D^2u_{\infty})=0 \hspace{0.2cm} \text{in}\hspace{0.2cm} B_{12\sqrt{d}}^+.
\]
Since $F_{\infty}$ is a $(\lambda,\Lambda)$-elliptic operator, we have $u_{\infty}\in C^{1,\alpha_0}(\overline{B}_{12\sqrt{d}}^+)$. By choosing $h\equiv u_{\infty}$, we get a contradiction with \eqref{8}.
\end{proof}

\begin{lemma}\label{lemma 7}
Let $0<\varepsilon_0< 1$, and $u\in C(\Omega)$ be a viscosity solution to
\begin{equation*}
\begin{cases}
F(D^2u,x)=f(x)&\hspace{.2in}\mbox{in}\hspace{.2in}B_{14\sqrt{d}}^+\\
u=0&\hspace{.2in}\mbox{on}\hspace{.2in} B_{14\sqrt{d}}\cap\{x_d=0\}.
\end{cases}
\end{equation*}
Suppose that A\ref{A-ellip}-A\ref{A-oscF} are in force and that 
\[
\|f\|_{L^p(B_{14\sqrt{d}}^+)} \leq\varepsilon
\]
If, 
\[
G_1(u,\Omega)\cap\left(\left(Q_2^{d-1}\times(0,2)\right)+\tilde{x}_0\right)\neq \emptyset,
\]
for some $\tilde{x}_0\in B_{9\sqrt{d}}\cap\{x_d\geq 0\}$, then there exists a universal constant $M$ such that
\[
\left|G_M(u,\Omega)\cap\left(\left(Q_1^{d-1}\times(0,1)\right)+x_0\right)\right|\geq 1-\varepsilon_0,
\]
where $x_0\in B_{9\sqrt{d}}\cap\{x_d \geq 0\}$.
\end{lemma}

\begin{proof}
The proof follows the general lines of Lemma \ref{lemma4}. We begin by setting
\[
v:=\frac{u-\ell}{C},
\]
where $\ell$ is an affine function and $C=C(d,\alpha)$ is chosen large enough, such that $\|v\|_{L^{\infty}(B_{13\sqrt{d}}^+)}\leq 1$ and
\[
-|x|^{1+\alpha}\leq v(x)\leq |x|^{1+\alpha} \text{ in } \Omega\setminus B_{13\sqrt{d}}^+.
\]
Moreover, $v$ solves
\[
\tilde{F}(D^2v,x)=\tilde{f}(x) \text{  in  } B_{14\sqrt{d}}^+,
\]
where
\[
\tilde{F}(D^2v,x):=\frac{1}{C}F(CD^2v,x) \hspace{0.2cm} \text{and}\hspace{0.2cm} \tilde{f}(x):=\frac{1}{C}f(x).
\]
Notice that the ellipticity constants of $F$ and $\tilde{F}$ are the same. Let $h\in C^{1,\alpha_0}_{loc}(\overline{B}_{13\sqrt{d}}^+)$ be the function from the previous lemma. The Maximum Principle implies that 
\[
\|h\|_{L^{\infty}(B^+_{13\sqrt{d}})}\leq \|v\|_{L^{\infty}(\partial B^+_{13\sqrt{d}})}\leq 1.
\]
Moreover, since $\|h\|_{C^{1,\alpha_0}(B^+_{12\sqrt{d}})}\leq C$, for some $N(d,\alpha)=N>1$, we have
\[
A_N(h,B_{12\sqrt{d}}^+)\cap\left(Q_1^{d-1}\times(0,1)+x_0\right)=\emptyset.
\]
Let us extend $h|_{B_{13\sqrt{d}}^+}$ continuously outside $B_{13\sqrt{d}}^+$ such that $h=v$ outside $B_{13\sqrt{d}}^+$. Hence 
\[
\|v-h\|_{L^\infty(\Omega)}=\|v-h\|_{L^\infty(B_{13\sqrt{d}}^+)}\leq 2,
\]
and
\[
-2-|x|^{1+\alpha}\leq h(x)\leq 2+|x|^{1+\alpha} \text{  in  } \Omega\setminus B_{13\sqrt{d}}^+.
\]
As a consequence, for some $M_0=M_0(d,\alpha)\geq N$, we have
\begin{equation}\label{M_0}
A_{M_0}(h,\Omega)\cap\left(\left(Q_1^{d-1}\times(0,1)\right)+x_0\right)=\emptyset.
\end{equation}
Now, set 
\[
w:=\frac{v-h}{2},
\] 
and notice that $w$ is a viscosity solution to 
\[
G(M,x):=\frac{1}{2}\tilde{F}(2M+D^2h,x)=\tilde{f}(x),
\]
and
\[
\|w\|_{L^\infty(\Omega)}=\|w\|_{L^\infty(B_{13\sqrt{d}}^+)}\leq 1.
\]
Thus, $w$ satisfies the hypotheses of the Lemma \ref{lemma 6} and we obtain for $t>1$:
\[
\left|A_t(w,\Omega)\cap\left(\left(Q_1^{d-1}\times(0,1)\right)+x_0\right)\right|\leq Ct^{-\mu}.
\]
In particular,
\[
\left|A_{2M_0}(v,\Omega)\cap\left(\left(Q_1^{d-1}\times(0,1)\right)+x_0\right)\right|\leq C{M_0}^{-\mu}.
\]
Finally, since by the definition of $v$ and Lemma \ref{GM properties}, $A_{2M_0}(v,\Omega)=A_{2CM_0}(u,\Omega)$, we set $M=2CM_0$ and finish the proof.
\end{proof}

\begin{lemma}\label{lemma 8}
Let $\varepsilon_0\in (0,1)$ and $u \in C(B^+_1)$ be a normalized viscosity solution to
\begin{equation*}
\begin{cases}
    F(D^2u,x)=f(x)&\hspace{.2in}\mbox{in}\hspace{.2in}B_{14\sqrt{d}}^+\\
    u=0&\hspace{.2in}\mbox{on}\hspace{.2in} B_{14\sqrt{d}}\cap\{x_d=0\}.
\end{cases}
\end{equation*}
Assume A\ref{A-ellip}-A\ref{A-oscF} are in force and $\|f\|_{L^d(B^+_{14\sqrt{d}})}\leq\varepsilon$. Extend $f$ by zero outside $B^+_{14\sqrt{d}}$ and set
\begin{align*}
A&:=A_{M^{k+1}}(u,B^+_{14\sqrt{d}})\cap\left(Q_1^{d-1}\times(0,1)\right)\\
B&:=\left(A_{M^{k}}(u,B^+_{14\sqrt{d}})\cap\left(Q_1^{d-1}\times(0,1)\right)\right)\cup\{x\in Q_1^{d-1}\times(0,1); M(f^d)\geq(C_0M^k)^d\},
\end{align*}
where $M>1$. Then $|A|\leq\varepsilon_0|B|$.
\end{lemma}

\begin{proof}
The proof is based on the Calderón-Zygmund cube decomposition, similar to Lemma \ref{lemma 5}. First, observe that since $M^{k+1} \geq M^k$, we have
$$
A_{M^{k+1}}(u, B_{14\sqrt{d}}^+) \subset A_{M^k}(u, B_{14\sqrt{d}}^+),
$$
which implies $A \subset B \subset \left(Q_1^{d-1} \times (0,1)\right)$.
$$
|A| \leq \sigma < 1,
$$
where $\sigma = \epsilon_0$.

Now, suppose that for some $i \geq 1$, 
\[
Q = \left( Q_{\frac{1}{2^i}}^{d-1} \times \left( 0, \frac{1}{2^i} \right) \right) + x_0
\]
is a dyadic cube with predecessor 
\[
\tilde{Q} = \left( Q_{\frac{1}{2^{i-1}}}^{d-1} \times \left( 0, \frac{1}{2^{i-1}} \right) \right) + \tilde{x}_0.
\]
Suppose by contradiction that $Q$ satisfies
\begin{equation}\label{eq_contradiction}
|A \cap Q| = |A_{M^{k+1}}(u, B_{14\sqrt{d}}^+) \cap Q| > \epsilon_0 |Q|,
\end{equation}
but $\tilde{Q} \nsubseteq B$. Hence, there exists a point $x_1 \in \tilde{Q} \setminus B$ such that
\begin{equation}\label{x_12}
x_1 \in \tilde{Q} \cap G_{M^k}(u, B_{14\sqrt{d}}^+),    
\end{equation}
and
$$
M(f^d)(x_1) < (C_0 M^k)^d.
$$
We split the proof into two cases based on the distance between $x_0$ and $(x'_0, 0)$.

\textbf{Case 1:} Suppose that
$$
|x_0 - (x'_0, 0)| < \frac{8}{2^i} \sqrt{d}.
$$
As before, we set
$$
T(y) := (x'_0, 0) + \frac{y}{2^i}.
$$
and define the rescaled function
$$
\tilde{u}(y) := \frac{2^{2i}}{M^k} u(T(y)).
$$
Notice that $\tilde{u}$ is a viscosity solution to the problem
$$
\begin{cases}
\tilde{F}(D^2 \tilde{u}, y) = \tilde{f}(y) & \text{in } B_{14\sqrt{d}}^+, \\
\tilde{u} = 0 & \text{on } B_{14\sqrt{d}} \cap \{x_d = 0\},
\end{cases}
$$
where
$$
\tilde{f}(y) := \frac{1}{M^k} f(T(y)), \quad \tilde{F}(N, y) := \frac{1}{M^k} F(M^k N, T(y)).
$$
Notice that $\tilde{F}$ also satisfies the same hypotheses as $F$, and following the same argument as in Lemma \ref{lemma 5}, we obtain
$$
\|\tilde{f}\|_{L^p(B_{14\sqrt{d}/2^i}^+)} \leq \varepsilon.
$$
From condition \eqref{x_12}, we also have
$$
G_1(\tilde{u}, T^{-1}(B_{14\sqrt{d}}^+)) \cap \left( \left(Q_2^{d-1} \times (0,2) \right) + 2^i (\tilde{x}_0 - (x'_0, 0))\right) \neq \emptyset.
$$
Moreover, since $ |x_0 - \tilde{x}_0| \leq \sqrt{d}/2^i$, it follows that
$$
|2^i (\tilde{x}_0 - (x'_0, 0))| \leq 9 \sqrt{d}.
$$
By Lemma \ref{lemma 7}, we conclude
$$
\left|G_M(\tilde{u}, T^{-1}(B_{14\sqrt{d}}^+)) \cap \left(\left( Q_1^{d-1} \times (0,1)\right) + 2^i (x_0 - (x'_0, 0))\right)\right| \geq 1 - \varepsilon_0.
$$
This implies
$$
|G_{M^{k+1}}(u, B_{14\sqrt{d}}^+) \cap Q| \geq (1 - \varepsilon_0) |Q|,
$$
which contradicts \eqref{eq_contradiction}.

\textbf{Case 2:} Now, suppose that
$$
|x_0 - (x'_0, 0)| \geq \frac{8}{2^i} \sqrt{d}.
$$
We define 
$$
T(y) := \left( x_0 + \frac{e_n}{2^{i+1}} \right) + \frac{y}{2^i},
$$
and
$$
v(y) := \frac{\tilde{u}(y)}{\|\tilde{u}\|_{L^\infty}}.
$$
We have that $v$ solves
$$
\begin{cases}
\tilde{F}(D^2 v, y) = \tilde{f}(y) & \text{in } B_{8\sqrt{d}}^+, \\
v = 0 & \text{on } B_{8\sqrt{d}} \cap \{x_d = 0\},
\end{cases}
$$
where $\tilde{f}$ and $\tilde{F}$ are defined analogously to {\bf Case 1}. Again, from \eqref{x_12}, we have
$$
G_1(v, T^{-1}(B_{8\sqrt{d}}^+)) \cap \left(\left( Q_3^{d-1} \times (0,3) \right) + 2^i (\tilde{x}_0 - (x'_0, 0))\right) \neq \emptyset.
$$
By applying \cite[Lemma 7]{Pimentel-Santos-Teixeira2022}, we can conclude that
$$
\left|G_M(v, T^{-1}(B_{14\sqrt{d}}^+)) \cap \left(\left( Q_1^{d-1} \times (0,1) \right) + 2^i (x_0 - (x'_0, 0))\right)\right| \geq 1 - \varepsilon_0,
$$
which again contradicts \eqref{eq_contradiction}.    
\end{proof}

\begin{proposition}\label{eq th1}
Let $u \in C(B_1)$ be a bounded viscosity solution to
\begin{equation}
\begin{cases}
F(D^2u,x)=f(x)&\hspace{.2in}\mbox{in}\hspace{.2in}B_1^+\\
u=0&\hspace{.2in}\mbox{on}\hspace{.2in} B_1\cap\{x_d=0\}.
\end{cases}
\end{equation}
Assume that A\ref{A-ellip}-A\ref{A-oscF} are satisfied. Then, $\theta\in L^p(B^+_1)$ and there exist a universal constant $C>0$ such that
\[
\|\theta\|_{L^p(B^+_{1/2})}\leq C.
\]
\end{proposition}

\begin{proof}
By performing a scaling (recall remark \ref{rem_scaling}), we can assume without loss of generality that $u$ is a viscosity solution to
\begin{equation*}
\begin{cases}
{F}(D^2{u},x)={f}(x)&\hspace{.2in} \mbox{in} \hspace{.2in} B_{14\sqrt{d}}^+\\
{u}=0&\hspace{.2in}\mbox{on}\hspace{.2in} B_{14\sqrt{d}}\cap\{x_d=0\},
\end{cases}
\end{equation*}
where the hypotheses $A$\ref{A-ellip}-$A$\ref{A-oscF} are also satisfied. Let $M$ and $C_0$ be the same as in Lemma \ref{lemma 8} and fix $\varepsilon_0$ so that $\varepsilon_0M^p = 1/2$. Now, we define
\[
\alpha_k:=|A_{M^k}({u},B^+_{14\sqrt{d}})\cap (Q^{d-1}_1\times (0,1))|,
\]
and
\[
\beta_k:=|\{x\in (Q^{d-1}_1\times(0,1)); M({f}^d)(x)\geq (C_0M^k)^d\}|.
\]
By applying Lemma \ref{lemma 8}, we obtain
\[
\alpha_{k+1}\leq\varepsilon_0(\alpha_k+\beta_k),
\]
which implies
\[
\alpha_k\leq\varepsilon^k_0+\sum^{k-1}_{i=0}\varepsilon^{k-i}_0\beta_i.
\]
We also have that $M({f}^d) \in L^{p/d}(B_{14\sqrt{d}})$ with
\[
\|M({f}^d)\|_{L^{\frac{p}{d}}(B^+_{14\sqrt{d}})}\leq C(p,d)\|{f}^d\|_{L^{\frac{p}{d}}(B^+_{14\sqrt{d}})}\leq\varepsilon C(d,p)\leq C(d,p),
\]
and since $\beta_k$ is the distribution function of $M(\tilde{f}^d)$ we get from Proposition \ref{Maximal function}
\[
\sum_{k\in\mathbb{N}}M^{pk}\beta_k\leq C\|{f}^d\|^{\frac{d}{p}}_{L^{\frac{p}{d}}(B^+_{14\sqrt{d}})} = C\|{f}\|^p_{L^p(B^+_{14\sqrt{d}})}\leq C(d,p).
\]
Moreover, 
\[
M^{pk}\alpha_k\leq M^{pk}\varepsilon^k_0+\sum^{k-1}_{i=0}\varepsilon^{k-i}_0M^{pk}\beta_k,
\]
so that
\[
\sum_{k\in\mathbb{N}}M^{pk}\alpha_k\leq\sum_{k\in\mathbb{N}}(M^p\varepsilon_0)^k+\sum_{k\in\mathbb{N}}\sum^{k-1}_{i=0}\varepsilon^{k-i}_0M^{p(k-i)}M^{pi}\beta_i\leq\sum_{k\in\mathbb{N}}2^{-k}+\left(\sum_{k\geq 0}M^{pk}\beta_k\right)\left(\sum_{k\in\mathbb{N}}2^{-k}\right)\leq C(d,p),
\]
which finishes the proof.
\end{proof}
Next, we present the proof of our first main result.
\begin{proof}[Proof of the theorem \ref{theo1}]
Let $x_0 \in B_{1/2}^+$ and $\psi$ be a $C^{1,\alpha}$-cone of opening $\pm M$ and vertex $x_0$ given by
\[
\psi(x)=L(x)\pm \frac{M}{2}|x-x_0|^{1+\alpha}.
\]
We estimate (extending $\psi$ by zero in a neighbourhood of $B_1^+$ if necessary)
{\small
\begin{align*}
    \Delta^{1+\alpha}_h\psi(x_0)&:=\dfrac{\psi(x_0+h)+\psi(x_0-h)-2\psi(x_0)}{|h|^{1+\alpha}}\\
    &=
    \dfrac{\ell(x_0+h)\pm \frac{M}{2}|x_0+h-x_0|^{1+\alpha}+\ell(x_0-h)\pm \frac{M}{2}|x_0-h-x_0|^{1+\alpha}-2(\ell(x_0)\pm\frac{M}{2}|x_0-x_0|^{1+\alpha})}{|h|^{1+\alpha}}\\
    &=
    \pm M.
\end{align*}
}
Now, consider the function $\tilde{u} := u\chi_{B_{1/2}^+}$, so we have $\tilde{u} \in L^\infty(\R^d)$ and $\tilde{u} \equiv u$ in $B_{1/2}^+$. 
If $\psi$ is a $C^{1, \alpha}$-cone of opening $M$ and vertex $x_0$, and 
touches $\tilde{u}$ strictly in $B_{1/10}(x_0)$ from above at $x_0 \in B^+_{1/2}$, then for all $h \in (0,\frac{1}{10})$ we have 
\begin{align*}
\Delta^{1+\alpha}_h\tilde{u}&=\dfrac{\tilde{u}(x_0+h)+\tilde{u}(x_0-h)-2\tilde{u}(x_0)}{|h|^{1+\alpha}}\\
& \leq\dfrac{\psi(x_0+h)+\psi(x_0-h)-2\psi(x_0)}{|h|^{1+\alpha}}\\
&\leq
\theta(\tilde{u},B_{1/2})(x_0)\\
&\leq
\theta(u,B^+_{1/2})(x_0).
\end{align*}
Similarly, if we touch $\tilde{u}$ in $B_{1/10}(x_0)$ from below at $x_0\in B^+_{1/10}$ by a concave $C^{1,\alpha}$-cone $\psi$ with opening $M$ and vertex $x_0$, then for all $h \in (0,\frac{1}{10})$:
\begin{align*}
\Delta^{1+\alpha}_h\tilde{u}&=\dfrac{\tilde{u}(x_0+h)+\tilde{u}(x_0-h)-2\tilde{u}(x_0)}{|h|^{1+\alpha}}\\
& \geq\dfrac{\psi(x_0+h)+\psi(x_0-h)-2\psi(x_0)}{|h|^{1+\alpha}}\\
&\geq
-\theta(\tilde{u},B_{1/2})(x_0)\\
&\geq
-\theta(u,B^+_{1/2})(x_0).
\end{align*}
Hence, the previous inequalities together with Proposition \ref{eq th1} yield to 
\[
\|\Delta^{1+\alpha}_h\tilde{u}\|_{L^p(B^+_{1/2})}\leq \|\theta(u,B^+_{1/2})\|_{L^p(B^+_{1/2})}\leq C,
\]
uniformly for all $0<h<1/10$, and for $0<\alpha<\alpha_0$. Let us define the singular integral operator
\[
I_{\sigma/2}(v)(x_0):=\int_{\mathbb{R}^d}\dfrac{v(x_0+y)+v(x_0-y)-2v(y)}{|y|^{d+\sigma}}dy,
\]
for $1 < \sigma < 1+\alpha$. Up to constants, we have that $I_{\sigma/2}(v)\approx(-\Delta)^{\sigma/2}(v)$. We estimate for $x_0\in B^+_{1/2}$,
\begin{align*}
I_{\sigma/2}(\tilde{u})(x_0)&=\int_{\mathbb{R}^d}\dfrac{\tilde{u}(x_0+y)+\tilde{u}(x_0-y)-2\tilde{u}(y)}{|y|^{d+\sigma}}dy=\int_{B_{1/2}}\dfrac{\tilde{u}(x_0+y)+\tilde{u}(x_0-y)-2\tilde{u}(y)}{|y|^{d+\sigma}}dy\\
&=
\int_{B_{1/10}}\dfrac{\tilde{u}(x_0+y)+\tilde{u}(x_0-y)-2\tilde{u}(y)}{|y|^{d+\sigma}}dy+\int_{B_{1/2}\setminus B_{1/{10}}}\dfrac{\tilde{u}(x_0+y)+\tilde{u}(x_0-y)-2\tilde{u}(y)}{|y|^{d+\sigma}}dy\\
&\leq
\theta(\tilde{u},B_{1/2})\int_{B_{1/10}}\frac{dy}{|y|^{d-\delta}}+C\|\tilde{u}\|_{L^{\infty}(B_1)}\\
&\leq
\dfrac{C}{\delta 10^d}\left(\theta(u,B^+_{1/2})+\|u\|_{L^{\infty}(B^+_1)}\right).
\end{align*}
where $C>0$ is a universal constant and $\delta=1+\alpha-\sigma$. Therefore,
\[
\|(-\Delta)^{\sigma/2}\tilde{u}\|^p_{L^p(B_{1/2})}\approx\|I_{\sigma/2}(\tilde{u})\|_{L^P(B_{1/2})}\leq \dfrac{C}{\delta 10^d}\left(\theta(u,B^+_{1/2})+\|u\|_{L^{\infty}(B^+_1)}\right),
\]
which implies
\[
(-\Delta)^{\sigma/2}\tilde{u} \in L^p(B^+_{1/2}).
\]
By setting $g:=(-\Delta)^{\sigma/2}\tilde{u}$ in $B^+_{1/2}$, we get that $\tilde{u}$ solve the following Dirichlet problem a.e. in $B^+_{1/2}$
\begin{equation}\label{eqfrac}
\begin{cases}
    (-\Delta)^{\sigma/2}\tilde{u}=g \quad &\text{in} \quad B^+_{1/2}\\
    \tilde{u}=0 \quad &\text{in} \quad  \mathbb{R}^d\setminus B^+_{1/2}.
\end{cases}
\end{equation}
Notice that since $u \in W^{1,p}({B_{1/2}^+})$, then $u \in W^{\sigma/2,p}(B_{1/2}^+)$, see for instance \cite[Proposition 2.2]{Nezza-Palatucci-Valdinoci-2012}.
Now, it is easy to see that $\tilde{u}$ is also a weak solution to \eqref{eqfrac}; to see that, let us consider $\varphi \in W^{\sigma/2,p}_0(B_{1/2}^+)$. Multiplying \eqref{eqfrac} by $\varphi$ and integrating over $\R^d$ we obtain
\[
\int_{\R^d}\varphi(x)(-\Delta)^{\sigma/2}u(x)dx = \int_{\R^d}\varphi(x)g(x)dx.
\]
We have that
\begin{align}
\int_{\R^d}\varphi(x)(-\Delta)^{\sigma/2}\tilde{u}(x)dx & = c(d,\sigma) \int_{\R^d}\int_{\R^d}\varphi(x)\dfrac{\tilde{u}(x) - \tilde{u}(y)}{|x-y|^{d+\sigma}}dydx \\
& = \dfrac{c(d,\sigma)}{2}\int_{\R^d}\int_{\R^d}\varphi(x)\dfrac{\tilde{u}(x) - \tilde{u}(y)}{|x-y|^{d+\sigma}}dydx - \dfrac{c(d,\sigma)}{2}\int_{\R^d}\int_{\R^d} \varphi(y)\dfrac{\tilde{u}(x) - \tilde{u}(y)}{|x-y|^{d+\sigma}}dydx \\
& = \dfrac{c(d,\sigma)}{2}\int_{\R^d}\int_{\R^d}\dfrac{(\tilde{u}(x) - \tilde{u}(y))(\varphi(x) - \varphi(y))}{|x-y|^{d+\sigma}}dydx,
\end{align}
which implies
\[
\dfrac{c(d,\sigma)}{2}\int_{\R^d}\int_{\R^d}\dfrac{(\tilde{u}(x) - \tilde{u}(y))(\varphi(x) - \varphi(y))}{|x-y|^{d+\sigma}}dydx = \int_{\R^d}\varphi(x)g(x)dx.
\]
Hence $\tilde{u}$ is a weak solution to \eqref{eq_fraclap}.
\end{proof}

\begin{proof}[Proof of Corollary \ref{cor_fracreg}]
Since $u$ is a weak solution to \eqref{eqfrac},  we can infer from \cite[Corollary 4.7]{AFLY2023} that $u \in W^{\gamma, p}(B_{1/2}^+)$, for all $\gamma < \sigma$. This finishes the proof.
\end{proof}


\section{Estimates for gradient-dependent operators}

In this section, we present the proof of Theorem \ref{prop geral case II}. Recall that from Remark \ref{rem_escau}, we can work under the weaker assumption $p \in (d-\varepsilon_0, +\infty)$.

\begin{proposition}\label{prop geral case I}
Let $u\in C(B^+_1)$ be a bounded viscosity solution to
\begin{equation}\label{eq prop geral case I}
    \begin{cases}
        F(D^2u,Du,u,x)=f(x) \quad &\text{in} \quad B^+_1,\\
        u=0 \quad & \text{on} \quad B_1\cap \{x_d=0\},
    \end{cases}
\end{equation}
Assume that A\ref{A-sourceterm}, A\ref{A-oscF} and A\ref{A-strcond} are in force and $\varepsilon \in (0, \alpha_0)$. Then there exists $\varepsilon_0 = \varepsilon_0(d,\lambda,\Lambda,b)$, such that if $p \in (d-\varepsilon_0, +\infty)$, then $u\in W^{\gamma,p}(B^+_{1/2})$ and
\begin{equation}\label{eq 2.16}
\|u\|_{W^{\gamma,p}(B^+_{1/2})}\leq C\left(\|u\|_{L^{\infty}(B^+_1)}+\|f\|_{L^p(B^+_1)}\right),
\end{equation}
for all $\gamma < 1+\varepsilon$, and for a positive constant $C=C(d,\lambda,\Lambda,b,c,p,r_0)$.
\end{proposition}

\begin{proof}

We prove that the problem can be reduced to considering only equations independent of $Du$ and $u$. By applying \cite[Theorem 3.6]{Caff-Crand-Kocan-Swi-1996}, we can infer that $u$ is pointwise twice differentiable a.e., and satisfies
$$
F(D^2u, Du, u, x) = f(x) \quad \text{a.e. in} \quad B_1^+.
$$
Let us define $\tilde{f}(x) := F(D^2u, 0, 0, x)$. From Assumption $A\ref{A-strcond}$, we have
$$
|\tilde{f}(x)| \leq b |Du(x)| + c |u(x)| + |f(x)| \quad \text{for a.e.} \quad x \in B_1^+.
$$
For a sufficiently small $\beta_0$ (recall $A\ref{A-oscF}$), it follows from the $W^{1,p}$-estimate (see \cite[Section 3]{nikki2009}) that $\tilde{f} \in L^p(B_1^+)$. Consequently, by applying \cite[Corollary 1.6]{Caff-Crand-Kocan-Swi-1996} we obtain
$$
F(D^2u, 0, 0, x) = \tilde{f}(x) \quad \text{in} \quad B_1^+,
$$
in the $L^p$-viscosity sense. Hence, we can consider that $u \in C(B_1^+)$ is a viscosity solution to
\begin{equation}\label{eq_101}
    \begin{cases}
        \tilde{F}(D^2u, x) = \tilde{f}(x) & \quad \text{in} \quad B_1^+, \\
        u = 0 & \quad \text{on} \quad B_1 \cap \{x_d = 0\},
    \end{cases}
\end{equation}
where $\tilde{F}(M, x) := F(M, 0, 0, x)$. At this point, we want to apply Theorem \ref{theo1}, but $F$ does not satisfy $A\ref{A-cont}$. To bypass this, we use an approximation procedure as in \cite[Proof of Theorem 4.3]{nikki2009}. Although we omit the detailed argument here (since it is very similar to the one presented in \cite{nikki2009}), we give some details below. We set  
\[
\tilde{F}_j(M,x) := \int \phi(x-y)\tilde{F}(M, y)dy,  
\]
for a suitable mollifier function $\phi$, and $F$ is extended by zero outside $B^+_1$, we can conclude that $F_j$ satisfies the hypotheses of Theorem \ref{theo1}, and there exists a family of solutions to \eqref{eq_101} $(u_j)_{j\in \N}$, such that $u_j \to u$ weakly in $W^{\gamma,p}(B^+_{1/2})$, for all $\gamma < 1+\varepsilon$, which implies (by the lower semicontinuity of the norm)
\[
\|u\|_{W^{\gamma,p}(B^+_{1/2})}\leq C\left(\|u\|_{L^{\infty}(B^+_1)}+\|{f}\|_{L^p(B^+_1)}\right).
\]
\end{proof}

In what follows, we prove our second main result, which is a generalization of the previous proposition for more general domains.

\begin{proof}[Proof of the Theorem \ref{prop geral case II}]
We begin by showing that it is enough to prove the result for the case where $\varphi \equiv 0$. To do this, define the function
\[
w := u - \varphi.
\]
Notice that $w$ is a viscosity solution to
\[
\begin{cases}
    G(D^2w, Dw, w, x) = g(x) & \text{in} \quad \Omega, \\
    w = 0 & \text{on} \quad \partial \Omega,
\end{cases}
\]
where $g(x) = f(x) - F(D^2\varphi(x), D\varphi(x), \varphi(x), x)$, and
\[
G(M, p, r, x) := F(M + D^2\varphi, p + D\varphi, r + \varphi, x) - F(D^2\varphi, D\varphi, \varphi, x).
\]
From the definition of $g$, and the structure condition on $F$ (recall A\ref{A-strcond}), we have that 
\[
|g| \leq |f| + C(\lambda, \Lambda,b,c)(\|D^2\varphi\|+\|Du\|+|\varphi|),
\] 
and thus $g \in L^p(\Omega)$. Since $F$ is uniformly elliptic, $G$ is also uniformly elliptic. 

Assuming that the theorem holds in this simplified case, we obtain
\[
\|w\|_{W^{\gamma,p}(\Omega)} \leq C\left(\|w\|_{L^{\infty}(\Omega)} + \|g\|_{L^p(\Omega)}\right),
\]
for all $\gamma < 1+\varepsilon$, which implies the desired estimate \eqref{eq prop geral case II}. From now on, we assume $\varphi \equiv 0$. Given that $\partial \Omega \in C^{1,1}$, we flatten the boundary in order to apply the previous proposition.

For any $y \in \partial \Omega$, there exists a $C^{1,1}$-diffeomorphism $\psi$ that maps a neighborhood $V(y)$ to the upper half-ball $B^+_1$, i.e.,
\[
\psi: V(y) \xrightarrow{\cong} B_1(0),
\]
such that $\psi(y) = 0$ and $\psi(V(y) \cap \Omega) = B^+_1$. Define $\vartheta := \phi \circ \psi \in W^{2,p}(V(y)) \cap C(V(y))$, where $\phi \in W^{2,p}(B^+_1) \cap C(B^+_1)$. Thus, we have
\[
D\vartheta = D\phi(\psi)D\psi \quad \text{and} \quad D^2\vartheta = D\psi^{T} D^2\phi(\psi) D\psi + D\phi(\psi) D^2\psi(\psi^{-1}(x)).
\]
We set $v := u \circ \psi^{-1} \in C(B^+_1)$. Since $u$ is a viscosity solution to \eqref{eq sol prop geral case II}, we obtain
\[
\tilde{F}(D^2 v(x), D v(x), v(x), x) := F(D^2(\vartheta(\psi^{-1}(x))), D(\vartheta(\psi^{-1}(x))), u(\psi^{-1}(x)), \psi^{-1}(x)).
\]
To simplify, we write it as
\[
\tilde{F}(D^2\phi(x), D\phi(x), v(x), x) = F(D^2\vartheta, D\vartheta, u, x) \circ \psi^{-1}.
\]
As a result, the function $v$ is a solution to
\[
\tilde{F}(D^2v, Dv, v, x) = \tilde{f}(x) \quad \text{in} \quad B^+_1,
\]
in the viscosity sense, where $\tilde{f}(x) := f(\psi^{-1}(x)) \in L^p(B^+_1)$. Observe that $\tilde{F}$ is uniformly elliptic, with possibly modified constants depending only on $\lambda, \Lambda$, and the $C^{1,1}$-norm of $\psi$.
To complete the proof, we verify that the condition on $\beta_{\tilde{F}}$ is satisfied. In fact, we have
\[
\beta_{\tilde{F}} := \frac{1}{\|M\|} |\tilde{F}(M, x) - \tilde{F}(M, y)| \leq C(\psi)\beta_F(\psi^{-1}(x), \psi^{-1}(y)) \leq \beta_0,
\]
for $\beta_0$ sufficiently small. Therefore, we can apply Proposition \ref{prop geral case I}, concluding the proof.
\end{proof}

\textbf{Acknowledgments.}  M. Santos was partially supported by the Portuguese government through FCT-Funda\c c\~ao para a Ci\^encia e a Tecnologia, I.P., under the projects UID/MAT/04459/2020, PTDC/MAT-PUR/1788/2020 and UIDB/04561/2020: https://doi.org/10.54499/UIDB/04561/2020 in CMAFcIO - University of Lisbon. 
C. Alcantara was partially supported by the Coordenação de Aperfeiçoamento de Pessoal de Nível Superior (CAPES) – Brasil and by Stone Instituição de Pagamento S.A., under the Project StoneLab. This study was financed in part by the CAPES—Brazil - Finance Code 001.

\section{Appendix}\label{appendix}

\subsection{Density results}

In this appendix, we show, without loss of generality, that we can assume $F$-harmonic functions belong to $W^{2,p}(B_1^+)$ (conform Remark \ref{rem_w2p}). Our main objective is to establish the following proposition:

\begin{proposition}\label{prop_w2papp}
Let $u \in C(B_1^+)$ be a viscosity solution to \eqref{1}, with $f \in L^p(B_1^+)$, $p > d$. Given $\delta > 0$, there exists a sequence $(u_n)_{n \in \N} \subset W^{2, p}(B^+_1)\cap S(\lambda - \delta, \Lambda+\delta, f)$ that converges locally uniformly to u.
\end{proposition}
The proof of Proposition \ref{prop_w2papp} follows from the ideas in \cite{Pimentel-Teixeira2016}, where the authors work with the recession operator, defined by $F^*(M, x):= \lim_{\mu \to \infty}F_\mu(M, x):= \lim_{\mu \to \infty}\mu F(\mu^{-1} M, x)$, where $\mu$ is a positive constant. For more details on the recession operator, see \cite{Pimentel-Teixeira2016} and \cite{silvtei}. We emphasize that such a notion is only used in the proof of Proposition \ref{prop_w2papp}, for a specific operator, and nowhere else. We assume that

\begin{Assumption}[$F^*$ has $C^{1,1}$-estimates]\label{A-recession}
We assume $F^*$ exists and that solutions to  
\begin{equation}
    \begin{cases}
        F^*(D^2u, x_0)=0 \quad &\text{in} \quad B_1 \,\,(\mbox{resp. } B_1^+),\\
        u=u_0 \quad & \text{on} \quad \partial B_1 \,\,(\mbox{resp. } \partial B_1^+),
    \end{cases}
\end{equation}
satisfies 
\[
\|u\|_{C^{1,1}(\bar{B}_{1/2})} \leq C\|u_0\|_{L^\infty(\partial B_{1})}\quad (\mbox{resp.}\quad \|u\|_{C^{1,1}(\overline{B^+}_{1/2})} \leq C\|u_0\|_{L^\infty(\partial B_{1}^+)}).
\]   
\end{Assumption}

\begin{Assumption}[Oscillation property]\label{A-betaf*}
We suppose that the oscillation of $F^*$ satisfies
\[
\dashint_{B_r}\beta_{F^*}(x, x_0)^d \leq Cr^{d\alpha},
\]
for $r \ll r_0$ and all $x_0 \in B_1$.
\end{Assumption}

\begin{remark}
    Notice that if $u$ is a solution to 
    \begin{equation*}
    \begin{cases}
        F(D^2u,x) = f(x) & \hspace{.2in} \mbox{in} \hspace{.2in} B_{14\sqrt{d}}^+\\
        u = 0 & \hspace{.2in} \mbox{on} \hspace{.2in} B_{14\sqrt{d}}\cap\{x_d=0\},
    \end{cases}
\end{equation*}
then, for any $\mu > 0$, the function $v := \mu u$ solves
\begin{equation*}
    \begin{cases}
        F_\mu(D^2v,x)=\tilde{f}(x) & \hspace{.2in} \mbox{in} \hspace{.2in} B_{14\sqrt{d}}^+\\
        v=0 & \hspace{.2in}\mbox{on} \hspace{.2in} B_{14\sqrt{d}}\cap\{x_d=0\},
    \end{cases}
\end{equation*}
where $\tilde{f}(x) := \mu f(x) \in L^p(B_{14\sqrt{d}})$.
\end{remark}

Under the conditions above, we can prove the following result:

\begin{theorem}\label{theo_recession_W2p}
Let $u \in C(B_1^+)$ be a viscosity solution to \eqref{1}, with $f \in L^p(B_1^+)$, $p > d$. Suppose that A\ref{A-ellip}, A\ref{A-recession} and A\ref{A-betaf*} hold true. Then $u \in W^{2,p}(B_{1/2}^+)$ and there exists a positive universal constant $C$ such that
\[
\|u\|_{W^{2,p}(B_{1/2}^+)} \leq C(\|u\|_{L^\infty(B_1^+)} + \|f\|_{L^\infty(B_1^+)}).
\]
\end{theorem}

The proof of Proposition \ref{theo_recession_W2p} proceeds in several steps, largely following the arguments of \cite[Section 4-6]{Pimentel-Teixeira2016}, combined with \cite[Section 2]{nikki2009}, with suitable modifications. Although we omit the detailed proof here, we provide some of its key ingredients for the convenience of the reader. The main modification to take into account is to replace the assumption $A$ of \cite{nikki2009}, namely, the local (resp. on the boundary) $C^{1,1}$-estimates for equations of the form
\begin{equation*}
    \begin{cases}
        F(D^2u, x_0)=0 \quad &\text{in} \quad B_1 \,\,(\mbox{resp. } B_1^+),\\
        u=u_0 \quad & \text{on} \quad \partial B_1 \,\,(\mbox{resp. } \partial B_1^+),
    \end{cases}
\end{equation*}
for the local (resp. on the boundary) $C^{1, 1}$-estimates for the recession operator, Assumption $A\ref{A-recession}$. We need the following Approximation Lemma.

\begin{lemma}[Approximation Lemma II]\label{lemma_appenapprox}
Let $u\in C(\overline{B}_{1}^+)$ be a normalized viscosity solution to
\begin{equation*}
    \begin{cases}
        F_\mu(D^2u,x)=f(x)&\hspace{.2in}\mbox{in}\hspace{.2in}B_{14\sqrt{d}}^+\\
        u=0&\hspace{.2in}\mbox{on}\hspace{.2in} B_{14\sqrt{d}}\cap\{x_d=0\}.
    \end{cases}
\end{equation*}
Assume that $A\ref{A-ellip}$, $A\ref{A-recession}$ and $A\ref{A-betaf*}$ are in force. Given $\delta>0$, there exists $\varepsilon >0$ such that if 
\[
\mu + \|f\|_{L^p(B_{14\sqrt{d}}^+)} <\varepsilon,
\]
then we can find a function $h\in C^{1,1}_{loc}(\overline{B}_{12\sqrt{d}}^+)$ satisfying
\[
\|u-h\|_{L^{\infty}(B_{12\sqrt{d}}^+)}\leq \delta.
\]
In addition, there exists a positive constant $C=C(d,\lambda,\Lambda,\alpha_0)$ such that
\[
\|h\|_{C^{1,1}(\overline{B}_{12\sqrt{d}}^+)}\leq C.
\]
\end{lemma} 

\begin{proof}
The proof follows the general lines of Proposition \ref{aproximationlemma}. Suppose the statement of the lemma is false, then we can find $\delta_0$ and sequences $(u_n)_{n\in \N}$, $(f_n)_{n\in \N}$ and $(\mu_n)_{n\in \N}$  satisfying
\[
 \mu_n + \|f_n\|_{L^p(B_{14\sqrt{d}}^+)} < \dfrac{1}{n},
\]
\[
F_{\mu_n}(D^2u_n, x) = f_n \quad \mbox{in} \quad B^+_{14\sqrt{d}},
\]
but
\begin{equation}\label{eq_apenapprox}
\|u_n - h\|_{L^\infty(B^+_{12\sqrt{d}})} > \delta_0,
\end{equation}
for all $n \in \mathbb{N}$ and $h \in C^{1, 1}(B^+_{12\sqrt{d}})$. As before, the regularity available for $u_n$ implies that we can find $u_\infty$ for which $u_n \to u_\infty$ locally uniformly in $B^+_{14\sqrt{d}}$. Moreover, as in \cite[Lemma 4.2]{Pimentel-Teixeira2016}, we can see that $u_\infty$ solves
\[
F^*(D^2u, x) = 0 \quad \mbox{in} \quad B^+_{13\sqrt{d}},
\]
which yields to $u_\infty \in C^{1,1}(B^+_{12\sqrt{d}})$. By taking $h \equiv u_\infty$ we reach a contradiction with \eqref{eq_apenapprox}.   
\end{proof}

The remainder of the proof of Theorem \ref{theo_recession_W2p} now closely follows the arguments in \cite{nikki2009}. More precisely, we continue as the analysis of \cite[Section 2]{nikki2009} with minor modifications, where for instance, in the proof of \cite[Lemma 2.14]{nikki2009} and \cite[Lemma 2.15]{nikki2009}, we use Lemma \ref{lemma_appenapprox}, instead of Lemma \cite[Lemma 2.13]{nikki2009}. 

Once we have Theorem \ref{theo_recession_W2p} available, the proof of Proposition \ref{prop_w2papp} follows as in \cite[Proposition 8.1]{Pimentel-Teixeira2016}, where we build an operator satisfying conditions $A$\ref{A-ellip}, $A\ref{A-recession}$ and $A\ref{A-betaf*}$, and then we use Theorem \ref{theo_recession_W2p}, instead of \cite[Theorem 6.1]{Pimentel-Teixeira2016} to finish the proof.

\subsection{Interior estimates}

\begin{theorem}[Fractional Sobolev Regularity]\label{prop interior estimate I}
Let $u \in C(B_1)$ be a bounded viscosity solution of
\begin{equation*}
\begin{cases}
    F(D^2u,x)=f(x) \quad &\text{in} \quad B_1\\
    u=0 \quad &\text{on} \quad \partial B_1.
\end{cases}
\end{equation*}
Assume that $A\ref{A-ellip}$-$A\ref{A-oscF}$ are in force. Then, we can find $g \in L^p(B^+_{1/2})$ such that the function $\tilde{u} := u\chi_{B_{1/2}}$ is a weak solution to
\begin{equation}
\begin{cases}\label{eq_fraclap}
    (-\Delta)^{\frac{\sigma}{2}}\tilde{u}=g \quad &\text{in} \quad B_{1/2},\\
    \tilde{u}=0 \quad &\text{in} \quad \R^d\setminus B_{1/2},
\end{cases}
\end{equation}
In particular, we have that $u \in W^{\sigma, p}(B_{1/2})$ with the estimate
\[
\|u\|_{W^{\sigma,p}(B_{1/2})} \leq C(\|u\|_{L^\infty(B_1)} +\|f\|_{L^p(B_1)}),
\]
where $\sigma = 1+\varepsilon$, for $\varepsilon\in (0,\alpha_0)$.
\end{theorem}

\begin{proof}
The proof of this theorem follows as in the arguments presented in \cite{Pimentel-Santos-Teixeira2022}. The key modification occurs in \cite[Lemma 5]{Pimentel-Santos-Teixeira2022}, where the result must be adjusted to accommodate operators with variable coefficients, under the additional assumption $A\ref{A-oscF}$. Specifically, if $u \in C(B_{8\sqrt{d}})$ is a normalized viscosity solution to
\[
F(D^2u, x) = f(x) \quad \text{in} \quad B_{8\sqrt{d}},
\]
then, given any $\delta > 0$, there exists a corresponding $\varepsilon > 0$ such that, provided $\|f\|_{L^p(B_{8\sqrt{d}})} \leq \varepsilon$, we can find a function $h \in C^{1,\alpha_0}_{\text{loc}}(B_{7\sqrt{d}})$ that satisfies
\begin{equation}\label{eq h delta-closy to u}
\|u - h\|_{L^{\infty}_{\text{loc}}(\overline{B}_{7\sqrt{d}})} \leq \delta.    
\end{equation}
We follow similar reasoning to prove this by contradiction as in the proof of \cite[Lemma 5]{Pimentel-Santos-Teixeira2022}. Assume there exists some $\delta_0 > 0$, along with a sequence of $(\lambda,\Lambda)$-elliptic operators $(F_n)$, and sequences of functions $(u_n)$ and $(f_n)$, such that
\[
F_n(D^2u_n) = f_n \quad \text{in} \quad B_1,
\]
and $\|f_n\|_{L^p(B_{8\sqrt{d}})} \leq 1/n$. However, for any function $h \in C^{1,\alpha_0}_{\text{loc}}(B_{7\sqrt{d}})$, we have
\[
\|u_n - h\|_{L^{\infty}(B_{6\sqrt{d}})} > \delta_0.
\]
From the regularity theory available for the family $(u_n)_{n \in N}$, combined with stability arguments, we deduce the existence of a function $u_{\infty} \in C^{\beta/2}_{\text{loc}}(B_{7\sqrt{d}})$ that solves the limit problem, which has $C^{1, \alpha_0}$-estimates. Setting $h \equiv u_{\infty}$ leads to a contradiction. Thus, we conclude that inequality \eqref{eq h delta-closy to u} holds true.

The remaining of the proof now closely follows sections 4 and 5 of \cite{Pimentel-Santos-Teixeira2022}, using \eqref{eq h delta-closy to u}, and accommodating the assumption $A\ref{A-oscF}$ in \cite[Lemma 6]{Pimentel-Santos-Teixeira2022}, \cite[Lemma 7]{Pimentel-Santos-Teixeira2022} and \cite[Lemma 8]{Pimentel-Santos-Teixeira2022} for instance.     
\end{proof}

The next proposition gives us an interior estimate for gradient-dependent operators.

\begin{proposition}\label{prop geral interior case I}
Let $u\in C(B^+_1)$ be a bounded viscosity solution to
\begin{equation}\label{eq prop interior geral case I}
    \begin{cases}
        F(D^2u,Du,u,x)=f(x) \quad &\text{in} \quad B_1,\\
        u=0 \quad & \text{on} \quad \partial B_1.
    \end{cases}
\end{equation}
Assume thatA\ref{A-sourceterm}, A\ref{A-oscF} and A\ref{A-strcond} hold true. Then there exists  $\varepsilon_0 = \varepsilon_0(d,\lambda,\Lambda,b)$, such that if $d-\varepsilon_0<p<+\infty$, then  $u\in W^{1+\varepsilon,p}(B^+_{1/2})$ and
\begin{equation}\label{eq 2.16}
\|u\|_{W^{1+\varepsilon,p}(B^+_{1/2})}\leq C\left(\|u\|_{L^{\infty}(B^+_1)}+\|f\|_{L^p(B^+_1)}\right),
\end{equation}
where $C=C(d,\lambda,\Lambda,b,c,p,r_0)$.
\end{proposition}

\begin{proof}
The proof is very similar to the proof of Proposition \ref{prop geral case I} and we omit it here.
\end{proof}

\bibliographystyle{amsplain}
\bibliography{Alcantara_Santos}

\bigskip

\noindent\textsc{Claudemir Alcantara}\\
Department of Mathematics\\
Pontifical Catholic University of Rio de Janeiro (PUC-Rio), \\
451-900, Rio de Janeiro, Brazil\\
\noindent\texttt{alcantara@mat.puc-rio.br}

\vspace{.15in}

\noindent\textsc{Makson S. Santos}\\
Departamento de Matem\'atica, Faculdade de Ci\^encias\\
Universidade de Lisboa\\
1749-016 Lisboa, Portugal\\
\noindent\texttt{msasantos@fc.ul.pt}

\end{document}